\documentclass[11pt,pdflatex]{amsart}

\usepackage{amssymb}
\usepackage{amsmath}
\usepackage{amsfonts}
\usepackage{hyperref}
\usepackage{multirow, comment}
\usepackage{mathrsfs}
\usepackage[mathscr]{euscript}

\theoremstyle{plain}
\newtheorem{theorem}{Theorem}[section]
\newtheorem{proposition}[theorem]{Proposition}
\newtheorem{corollary}[theorem]{Corollary}
\newtheorem{lemma}[theorem]{Lemma}

\theoremstyle{definition}
\newtheorem{remark}[theorem]{Remark}

\newcommand{\epf}{{\ifhmode\unskip\nobreak\hfil\penalty50 \hskip1em
\else\nobreak\fi \nobreak\mbox{}\hfil\mbox{$\square$} \parfillskip=0pt
\finalhyphendemerits=0 \par\vskip5pt}}

\newcommand{\lra}{\longrightarrow}

\newcommand{\CC}{\mathbf{C}}
\newcommand{\FF}{\mathbf{F}}

\newcommand{\QQ}{\mathbf{Q}}
\newcommand{\RR}{\mathbf{R}}
\newcommand{\ZZ}{\mathbf{Z}}

\newcommand{\ar}{\mathrm{ar}}

\newcommand{\CO}{\mathcal{O}}

\newcommand{\gp}{\mathfrak{p}}

\newcommand{\Fpbar}{{\overline{\FF}_p}}
\usepackage[usenames]{color}
\usepackage{color}
\usepackage{amssymb}
\usepackage{graphicx, epsfig}
\usepackage{latexsym, amsfonts, amscd, amsmath}
\usepackage{mathrsfs}

\makeindex 
\input xy
\xyoption{all}

\definecolor{orange}{rgb}{1,0.5,0}
\definecolor{Indigo}{rgb}{0.2,0.1,0.7}
\definecolor{Violet}{rgb}{0.5,0.1,0.7}

\theoremstyle{plain}
\newtheorem{thm}[theorem]{Theorem}
\newtheorem{prop}[theorem]{Proposition}

\newtheorem{cor}[theorem]{Corollary}

\theoremstyle{definition}
\newtheorem{dfn}[theorem]{Definition}

\theoremstyle{remark}

\newenvironment{mat}{\left(\begin{smallmatrix}
\end{smallmatrix}\right)}{enddef}

\numberwithin{equation}{section}
\numberwithin{figure}{section} \numberwithin{table}{subsection}

\newcommand{\Hom}{{\operatorname{Hom}}}

\newcommand{\Res}{{\operatorname{Res }}}

\newcommand{\Spec}{{\operatorname{Spec }}}

\newcommand{\GL}{{\operatorname{GL}}}

\newcommand{\gerd}{{\frak{d}}}

\newcommand{\gerp}{{\frak{p}}}
\newcommand{\gerq}{{\frak{q}}}

\newcommand{\calA}{{\mathcal{A}}}

\newcommand{\calC}{{\mathcal{C}}}
\newcommand{\calD}{{\mathcal{D}}}

\newcommand{\calG}{{\mathcal{G}}}
\newcommand{\calH}{{\mathcal{H}}}

\newcommand{\calL}{{\mathcal{L}}}

\newcommand{\calO}{{\mathcal{O}}}

\newcommand{\calT}{{\mathcal{T}}}

\newcommand{\calV}{{\mathcal{V}}}

\newcommand{\calZ}{{\mathcal{Z}}}

\def\AA{\mathbb{A}}
\def\BB{\mathbb{B}}
\def\CC{\mathbb{C}}

\def\FF{\mathbb{F}}
\def\GG{\mathbb{G}}

\def\PP{\mathbb{P}}
\def\QQ{\mathbb{Q}}
\def\RR{\mathbb{R}}

\def\TT{\mathbb{T}}

\def\ZZ{\mathbb{Z}}

\newcommand{\ra}{{\; \rightarrow \;}}

\newcommand{\Qpbar}{\overline{\QQ}_p}

\begin{document}

\title[Cones of Weights and Minimal Cones of Strata]{Cones of Weights and Minimal Cones of the Goren-Oort Strata in Hilbert modular varieties}

\author{Fred Diamond}
\email{fred.diamond@kcl.ac.uk}
\address{Department of Mathematics,
King's College London, WC2R 2LS, UK}
\author{Payman L Kassaei}
\email{payman.kassaei@kcl.ac.uk}
\address{Department of Mathematics,
King's College London, WC2R 2LS, UK}

\date{April 2025}

\begin{abstract}  Let $p$ be a prime, $F$ a totally real field in which $p$ is unramified, and $X/\Fpbar$ a Shimura variety associated to ${\rm Res}_{F/\QQ} {\rm GL}_2$ (or a PEL Hilbert modular variety).  A mod $p$ Hilbert modular form of weight $\kappa$ can be defined as a section of an automorphic line bundle $\calL_\kappa$ on $X$. We consider sections of $\calL_\kappa$ (forms) over a Goren-Oort stratum $X_T$ inside $X$, and define the cone of weights of $X_T$ to be the $\QQ^{\geq 0}$-cone generated by the weights of all nonzero  forms on $X_T$. We explicitly determine the cone of weights of all strata, showing in particular that they are not in general generated by the weights of the associated Hasse invariants. Using this, we define a notion of minimal cone for each stratum, and explicitly determine the minimal cones of all strata. When $X$ is a Shimura variety associated to ${\rm Res}_{F/\QQ} {\rm GL}_2$, we prove that for every nonzero eigenform $f$ for the prime-to-$p$ Hecke algebra on a stratum $X_T$, there is another eigenform with the same Hecke eigenvalues which has weight in the minimal cone of $X_T$.
\end{abstract}
\maketitle

\section{Introduction}
Let $F$ be a totally real field of degree $d>1$ over $\QQ$, $\calO_F$ its ring of integers, and $p$ a prime which is unramified in $F$. Recall that the weight of a mod $p$ (geometric) Hilbert modular form is a $d$-tuple of integers $\kappa=\sum_{\beta \in \BB} k_\beta e_\beta \in \ZZ^\BB$, where $\BB=\Hom(\calO_F,\Fpbar)$. Let $\sigma$ denote the Frobenius automorphism acting on $\BB$ by composition on the right. In \cite{DK1}, we introduced the notion of the minimal cone of weights for mod $p$ Hilbert modular forms:
\[
\calC^{\rm min}=\{\kappa=\sum_{\beta \in \BB} k_\beta e_\beta \in \QQ^\BB: pk_\beta \geq k_{\sigma^{-1}\beta}, \forall \beta \in \BB\},
\]
and proved that for every mod $p$ Hilbert modular eigenform $f$, there is an eigenform with the same $q$-expansion which has   weight in $\calC^{\rm min}$. In fact, we proved a stronger result:   every mod $p$ Hilbert modular form that is not divisible by any partial Hasse invariants has weight in  $\calC^{\rm min}$.  An immediate consequence of this result is that the cone of weights of mod-$p$ Hilbert modular forms, i.e., the $\QQ^{\geq 0}$-cone generated by the weights of nonzero Hilbert modular forms, is generated by the weights of the partial Hasse invariants, namely, it is the Hasse cone: 
\[
\calC^{\rm Hasse}=\{\kappa\in \QQ^\BB: \kappa= \sum_{\beta \in \BB} x_\beta h_\beta\ {\rm such\ that}\ x_\beta \in \QQ^{\geq0},  \forall \beta\in \BB\},
\]
where, $h_\beta =-e_\beta + p e_{\sigma^{-1}\beta}$ is the weight of the partial Hasse invariant ${\bf h}_\beta$. In \cite{DK2}, we extended these results to the general case where $p$ may be ramified in $F$. 

As part of a program to study the cone of weights of automorphic forms on Shimura varieties, Goldring and Koskivirta \cite{GK} also showed, using different methods,  that when $p$ is unramified in $F$, the cone of weights of Hilbert modular forms equals the Hasse cone. One question that arises in the work of Goldring and Koskivirta is to determine Shimura varieties for which the cone of weights is generated by the weights of the relevant Hasse invariants. While, as mentioned above, this is the case over a Hilbert modular variety, it fails to hold over some of its Goren-Oort strata. The central idea in this work is that the generators of the cone of weights should be sought among pullbacks of Hasse invariants of certain Shimura varieties to which the space of interest maps. We will make this precise. 

Let $X/\Fpbar$ denote a Hilbert modular variety of level prime to $p$, and for $\kappa=\sum_{\beta \in \BB} k_\beta e_\beta$, let $\omega^\kappa$ be the corresponding automorphic line bundle.  Recall that if $\pi:\calA \rightarrow X$ is the universal abelian scheme over $X$, then $\pi_\ast\Omega_{\calA/X}\cong \oplus_{\beta\in \BB}\omega_\beta$, where, $\omega_\beta$ is a line bundle, $\calO_F$ acts on it via  $\beta$, and $\omega^\kappa=\otimes_{\beta\in\BB}\!\ \omega_\beta^{k_\beta}$. Recall further the partial Hasse invariant ${\bf h}_\beta \in H^0(X,\omega^{h_\beta})$ defined in \cite{AG}. For any $T \subseteq \BB$, the Goren-Oort stratum $X_T$ is defined by the vanishing of partial Hasse invariants ${\bf h}_\beta$, for $\beta \in T$. By the cone of weights of $X_T$, denoted $\calC(X_T,\omega^\bullet)$, we mean the $\QQ^{\geq 0}$-cone generated by those $\kappa$ for which $H^0(X_T,\omega^\kappa)\neq 0$. The collection of the Hasse invariants of $X_T$ is given by  $\{ {\bf h}_\beta: \beta \not \in T; {\bf b}_\beta: \beta \in T\},$ where, for $\beta \in T$,  ${\bf b}_\beta$ is a nowhere vanishing section of weight $b_\beta=e_\beta+pe_{\sigma^{-1}\beta}$ on $X_T$, corresponding to the isomorphism $\omega_\beta \cong \omega^{-p}_{\sigma^{-1}\beta}$. In particular, we always have an inclusion of $\QQ^{\geq 0}$-cones
\begin{eqnarray}\label{eqn: inclusion of cones}
\langle h_\beta: \beta \not\in T; \pm b_\beta: \beta \in T\rangle \subseteq \calC(X_T,\omega^\bullet).
\end{eqnarray}
Goldring and Koskivirta \cite{GK} showed that the above inclusion is an equality for certain ``admissible" strata, a result that also follows from our work in \cite{DK1}, and asked the question if this is the case for all strata. In the first part of this paper, we show that the inclusion is proper for all ``non-admissible" strata, and determine the cone of weights for all strata. In fact, we show that  every generator of $\calC(X_T,\omega^\bullet)$ is obtained by pulling back a Hasse invariant over another stratum $X_{T'}$, via a map $X_T\rightarrow X_{T'}$. Let us make this more explicit. For simplicity of presentation, we will assume $p$ is inert in $F$, and consider only $T \subsetneq \BB$.  In \cite {TX}, Tian and Xiao prove that for such $T$, there is a $T \subseteq \tilde{T}$ of even cardinality, such that $X_T$ is isomorphic to a $(\PP^1)^{|\tilde{T}-T|}$-bundle over a quaternionic Shimura variety associated to the quaternion algebra $B_{\tilde{T}}$ that is ramified exactly at $\tilde{T}$, where $\tilde{T}$ is considered as a subset of $\Hom(F,\RR)$. In particular, this implies that if $T \subseteq T' \subseteq \BB$ are such that $\tilde{T}=\tilde{T}'$, then there is a map\footnote{ These maps could also be constructed directly using Dieudonn\'e module theory.} $\pi_{T,T'}:X_{T} \rightarrow X_{T'}$ which makes $X_{T}$ a  $(\PP^1)^{|T'-T|}$-bundle over $X_{T'}$. We prove that $\calC(X_T, \omega^\bullet)$ is generated by the weights of the pullbacks of the Hasse invariants of $X_{T'}$ under $\pi_{T,T'}$ as $T'$ varies subject to the above condition. More precisely, for the restricted choice of $T$ as above, we have
\[
\calC(X_T, \omega^\bullet)=\langle f_\beta(T): \beta \in \BB-T; \pm b_\beta: \beta \in T\rangle,
\]  
where $f_\beta(T)=-e_{\sigma^{n_\beta}\beta}-(-p)^{n_\beta}e_\beta$, where $n_\beta$ is the smallest $i>0$ such that $\sigma^i\beta \not\in \tilde{T}$, if $\tilde{T}\neq \BB$, and $n_\beta=d$ if $\tilde{T}=\BB$. When $\tilde{T}=T$, i.e., when $X_T$ is isomorphic to a quaternionic Shimura variety associated to $B_{\tilde{T}}$, this cone equals the cone of weights generated by the weights of the Hasse invariants of $X_T$ (as well as those of the quaternionic Shimura variety). This is the admissible case referred to above. In all other cases, the inclusion \ref{eqn: inclusion of cones} is proper.

In the second part of the article, we study the question of minimal weights of forms over strata. One motivation for this work has been to clarify what a minimal cone of weights should mean for more general Shimura varieties, and how such minimal cones can be calculated. In \cite{DK1}, the definition of the minimal cone was motivated by considerations arising from the relationship between Hilbert modular forms and Galois representations as in \cite{DS}. In this work, we show how the minimal cone of a stratum $X_T$ can be defined using a partial subset of the generators $\kappa$ of $\calC(X_T,\omega^\bullet)$, for which $ H^0(C,\omega^\kappa)$ is $1$-dimensional for any connected component $C$ of $X_T$.     Starting with this definition, we go on to explicitly calculate the minimal cones of all strata, and prove that  these  minimal cones play the same role vis-\`a-vis the minimal weights of Hecke eigenforms over the strata. To account for the (prime-to-$p$ Hecke action), in the final part of this article, we study the same questions over Shimura varieties associated to ${\rm Res}_{F/\QQ} {\rm GL}_2$ (as opposed to the related PEL Shimura varieties implicitly used in the discussion above). We will make this more precise below. 

Let $X'/\Fpbar$ denote the Shimura variety associated to ${\rm Res}_{F/\QQ} {\rm GL}_2$ of a level prime to $p$. Again, for any $T \subset \BB$, there is a stratum $X'_T$ given by the vanishing of the partial Hasse invariants indexed by $\beta \in T$.  In this setting, the weight of a form defined over any stratum is a tuple $(\lambda, \kappa)$, where $\kappa=\sum_\beta k_\beta e_\beta$, $\lambda=\sum_\beta \ell_\beta e'_\beta$ are elements of  $\ZZ^\BB$. Defining the cone of weights $\calC(X'_T,\delta^\bullet\omega^\bullet)$ to be the $\QQ^{\geq 0}$-cone generated by weights of nonzero forms on $X'_T$, we show that $$\calC(X'_T,\delta^\bullet\omega^\bullet)=\langle f_\beta: \beta \in \BB-T; \pm b_\beta: \beta \in T; \pm e'_\beta: \beta \in \BB\rangle$$
We further define and explicitly describe the minimal cone of weights for $X'_T$. Finally, we  prove that for any form ${\bf f}$ of weight $(\kappa,\lambda)$ on $X'_T$ which is an eigenform for the prime-to-$p$ Hecke algebra, there is a from ${\bf f}_0$ on $X'_T$ which has the same Hecke eigenvalues as ${\bf f}$, such that the weight of ${\bf f}_0$  lies in the minimal cone of $X'_T$.

While, in \cite{DK1,DK2}, we used the explicit description of the minimal cone of $X$ to identify the cone of weights on $X$, in this work, we first determine the cone of weights for all strata in $X$, and then use that to find an explicit description of the minimal cones. We propose this as a strategy for describing minimal cones over more general Shimura varieties, in combination with the methods developed by Goldring, Koskivirta and their collaborators.

As stated earlier, every stratum $X_T$ is Hecke-equivariantly isomorphic to a flag space over some quaternionic Shimura variety. We would like to mention that, while we don't pursue this in the article, the exact methods of this work can be used to determine cones of weights and minimal cones of flag spaces over quaternionic Shimura varieties more generally, (and, in particular, those of the full flag spaces). For example, given $T \subset \BB$, the cone of weights of the full flag space over a Shimura variety associated to $B_{\tilde{T}}$  is generated by all $f_\beta(T)$ for $\beta \in \BB$, and the explicit description of the minimal cone is similar to those obtained in this paper.

Finally, we note that Gabriel Micolet \cite{M} has recently generalized the main result of \cite{TX} to the case where $p$ may be ramified in $F$, using an approach as in \cite{DKS}.  Using the constructions of \cite{M}, we expect the results of this paper can be generalized to the case where $p$ is ramified in $F$.

 \subsection{Acknowledgements}  We are grateful to Wushi Goldring, Jean-Stefan Koskivirta, and Deding Yang for several helpful conversations around the themes of this paper. This research also benefited from discussions sparked by a workshop on Serre weight conjectures held at King's College London in March 2024, sponsored by the Heilbronn Institute for Mathematical Research.

 \section{Hilbert modular varieties}\label{section: HMV}

 Let $F$ be a totally real field of degree $d>1$, $\calO_F$ its ring of integers, and $\gerd_F$ its different. Let $p$ be a prime number which is unramified in $F$. For a  prime ideal $\gerp$  of $\calO_F$ dividing $p$, let $\FF_\gerp = \calO_F/\gerp$, and denote by  $\FF$ a finite field containing an isomorphic copy of each $\FF_\gerp$. Let $\QQ_\FF$ denote the fraction field of $W(\FF)$, and fix embeddings $\iota_\FF: \mathbb{Q}_\FF \hookrightarrow \QQ_p^{\rm ur}
 \subset \Qpbar$, as well as embeddings $\iota_p:\overline{\QQ} \hookrightarrow \overline{\QQ}_p $ and
 $\iota_\infty: \overline{\QQ} \hookrightarrow \CC$. Let $\AA^\infty_F$ denote the finite adeles over $F$, and 
$\AA_F^{\infty, p}$ the finite prime-to-$p$ adeles of $F$.

Let $\BB={\rm Emb}(F,\mathbb{Q}_\FF)$.  We write $\BB=\textstyle\coprod_{\mathfrak{p}|p}
\mathbb{B}_\mathfrak{p}$, where 
$
\mathbb{B}_\mathfrak{p}= \{\beta\in\mathbb{B}\colon
\gerp \subseteq \beta^{-1}(pW(\FF))\}.
$
 Using $\iota_p$, $\iota_\infty$, we identify  $\BB$ with the set of infinite places of $F$, and think of $\BB_\gerp$ as those inducing the prime ideal $\gerp$. Let $\sigma$ denote the Frobenius automorphism of $\mathbb{Q}_\FF$, lifting $x \mapsto
x^p$ modulo $p$. It acts on $\BB$ by $\beta \mapsto \sigma
\circ \beta$. It acts transitively on each $\BB_\mathfrak{p}$.

Let $K$ be a small enough (see \cite[Lemma~2.4.1]{DS} for example) open compact subgroup of ${\rm GL}_2(\AA_F^{\infty})$ containing ${\rm GL}_2(\calO_{F,p})$. Write $K=K^pK_p$, where $K^p \subset {\rm GL}_2(\AA_F^{\infty, p})$, and $K_p={\rm GL}_2(\calO_{F,p})$. Fix an element $\epsilon \in (\AA_F^{\infty,p})^\times/\det(K^p)(\widehat{\ZZ}^{(p)})^\times$, and let $I$ denote the (prime-to-$p$) fractional ideal of $F$ generated by $\epsilon$.

We denote by $X_K^\epsilon$ the scheme of relative dimension $d$ over $\Spec(\ZZ_{(p)})$ which represents the functor:
\[
 \mathcal{M}: \langle{\rm Schemes}/\ZZ_{(p)} \rangle \ra \langle{\rm Sets}\rangle,
 \] where $\mathcal{M}(S)$ is the set of all isomorphism classes of  $\underline{A}/S=(A/S,\iota,\lambda,\alpha)$ such that:  
 \begin{itemize}
 \item $A$ is an abelian scheme of relative dimension $d$ over  $S$;
\item $\iota\colon\mathcal{O}_F \hookrightarrow {\rm End}_S(A)$ is a ring
homomorphism such that ${\rm Lie}(A/S)$ is locally free of rank one over $\calO_F \otimes \calO_S$;
 \item  $\lambda$ is a quasi-polarization inducing an $\CO_F$-linear isomorphism $A\otimes_{\CO_F} \gerd_F^{-1} \stackrel{\sim}{\longrightarrow} A^{\vee} \otimes_{\CO_F} I$ (in particular, $\iota$ is invariant under the associated Rosati involution);
 \item  $\alpha$ is a level $K^p$-structure compatible with $\epsilon$ (more precisely, in the sense of \S2.1.2 of \cite{DKS}, with compatibility meaning that the diagram on p.12 of {\em loc.~cit.}~commutes for some $\zeta$).
\end{itemize}

\begin{remark} There are two standard notions of Hilbert modular varieties: Shimura varieties associated to the group $\Res_{F/\QQ}\GL_2$, or to the group $G^*: = \Res_{F/\QQ}\GL_2 \times_{T_F} \GG_m$, where $T_F = \Res_{F/\QQ}\GG_m$ (the morphisms to it being induced by the determinant and base-change).  The advantage of the latter setting is its representability by PEL moduli problems, whereas the former is more readily applicable in the context of Langlands reciprocity. Our scheme $X_K^{\epsilon}$ is a model for the Shimura variety of level $K^* = g^{-1}Kg \cap G^*(\AA_F^\infty)$ for the group $G^*$, where $g \in \GL_2(\AA_F^{\infty,p})$ is such that $\det(g) = \epsilon$, with conventions chosen as in \cite{DKS} to facilitate the transfer of results (in \S\ref{section: GL2} below) to the setting of $\Res_{F/\QQ}\GL_2$.  However, we differ slightly from the set-up in \cite{DKS} by imposing the compatibility condition with a fixed $\epsilon$, as it will be more convenient here to work with (not just locally) Noetherian schemes.
\end{remark}

We let $X=X_K^\epsilon \otimes_{\ZZ_{(p)}} \FF$. Let $s:\calA \rightarrow X$ denote the universal abelian scheme. Let $\omega= s_*\Omega^1_{\calA/X}$ and $\calH=s_*H^1_{\rm dR}(\calA/X)$.
By assumption,   $\omega$ is locally free of rank one over  $\calO_{X} \otimes_\ZZ \calO_F \cong \oplus_{_{\beta\in \BB}}\  \CO_{X}$.
This induces a decomposition
\[
\omega=\oplus_{_{\beta\in \BB}}\  \omega_{\beta},
\]
where, each $\omega_\beta$ is locally free of rank one over $\calO_{X}$.
Similarly, we have a decomposition
\[
\calH =\oplus_{_{\beta\in \BB}}\  \calH_{\beta},
\]
where, each  $\calH_\beta$ is locally free of rank two over $\calO_{X}$.  Furthermore, using the quasi-polarization $\lambda$ and the canonical isomorphism $R^1\epsilon_*\CO_{\calA^\vee} \cong{\rm Lie}(\calA/X)$, we obtain isomorphisms $\epsilon_\beta: \wedge^2 \calH_\beta \cong \calO_{X}$.

 For $\beta \in \BB$, we let $e_\beta \in \ZZ^\BB$ correspond to $\beta$. For any $\kappa = \sum k_\beta e_\beta \in \ZZ^\BB$, we define
\[
\omega^\kappa=\otimes_{_{\beta \in \BB}}\  \omega_\beta^{k_\beta}.
\]
For each $\beta \in \BB$, there is a partial Hasse invariant ${\bf h}_\beta \in H^0(X,\omega_{\sigma^{-1}\beta}^{\otimes p}\otimes \omega_\beta^{-1})$ defined in \cite{AG}, as the $\beta$-component of the Verschiebung map $V:\omega \rightarrow \omega^{(p)}$.

The Goren-Oort stratification on $X$ is a collection $\{X_T\}_{T \subseteq \BB}$ of closed subschemes of $X$, where $X_T$, given by the vanishing of $\{{\bf h}_\beta: \beta \in T\}$, is a smooth and equi-dimensional closed subscheme of codimension $|T|$. If $\beta \not\in T$, we denote the restriction of ${\bf h}_\beta$ to $X_T$ by ${\bf h}_\beta(T)$. 

\begin{dfn}\label{dfn: b} Let $\gerp|p$, and $\beta\in \BB_\gerp$. Assume $\beta \in T \subset \BB$. Consider the Verschiebung morphism   $V_{\beta}: \calH_{\beta} \rightarrow \calH^{(p)}_{\sigma^{-1}\beta}$ restricted to $X_T$. It is surjective onto $\omega_{\sigma^{-1}\beta}^{(p)}$, and  since $\beta \in T$, it vanishes on $\omega_{\beta}$,  inducing an isomorphism of line bundles on $X_T$:
$$ V_{\beta}: \omega_{\beta}^{-1}  \overset{\epsilon_\beta}{\cong} \omega_{\beta}^{-1} \otimes  \wedge^2 \calH_\beta  \cong \calH_{\beta}/\omega_{\beta} \xrightarrow{\sim}  \omega_{\sigma^{-1}\beta}^p.
$$
We let ${\bf b}_\beta(T)$ denote  the corresponding nowhere-vanishing section in $H^0(X_T, \omega_{\sigma^{-1}\beta}^{\otimes p}\otimes \omega_\beta)$.  \end{dfn}
 
 \begin{cor}\label{cor: torsion} Let $\BB_\gerp \subset T \subset \BB$. Then, for each $\beta \in \BB_\gerp$, we have $\omega_\beta^{(-p)^{|\BB_\gerp|}-1} \cong \calO_{X_T}$.
 \end{cor}
   
\section{Morphisms between strata}\label{section: morphisms}

We start by reviewing some results from \cite{TX}. Let $T=\sqcup_{\gerp|p} T_\gerp$, where $T_\gerp \subseteq B_\gerp$.     A chain in $T_\gerp$ is a subset of the form $C=\{\sigma^{-i}\beta: 0 \leq i \leq m\}$ where $m\geq 0$, and so that $\sigma^{-m-1} \beta \not \in T_\gerp$ and $\sigma\beta \not\in T_\gerp$. For such a chain, we  define $\tilde{C}=C$ if $m$ is odd, and $\tilde{C}=C\cup\{\sigma^{-m-1} \beta\}$ if $m$ is even. Any $T_\gerp\subsetneq \BB_\gerp$ can be written as a disjoint union of chains, $T_\gerp=\sqcup_j C_j$, and we define $\tilde{T}_\gerp=\sqcup_j \tilde{C_j}$. If $T_\gerp=\BB_\gerp$, we define $\tilde{T}_\gerp=\BB_\gerp$. Set $\tilde{T}=\sqcup_{\gerp|p} \tilde{T}_\gerp$, and $$S(T)=\tilde{T}\cup \{\gerp: T_\gerp=\BB_\gerp, 2\!\not |\ |\BB_\gerp|\}.$$ Then $S(T)$ is a set of places of $F$ of even cardinality, and we denote by $B_{S(T)}$ the quaternion algebra over $F$ ramified exactly at $S$.  We define 
$${\rm Iw}(T)=\{\gerp|p: T_\gerp=\BB_\gerp,\ {\rm and}\ |\BB_\gerp|\ {\rm is\ even}\}.$$

Let $G$ denote the algebraic group over $\QQ$ defined by 
$$G(R) = \{\,g\in (B_{S(T)}\otimes R)^\times\,|\,\nu(g) \in R^\times\,\},$$ where $\nu$ is the reduced norm, and
consider the system of quaternionic Shimura varieties $\{X_{G,K_T}\}_{K_T}$ for $G$, where $K_T\subset G(\AA^{\infty})$ is a compact open subgroup.   Fix $T \subseteq \BB$, and consider all $T'$ satisfying $T \subseteq T' \subseteq \tilde{T}$, such that ${\rm Iw}(T')={\rm Iw}(T)$. Note that $S(T)=S(T')$. It follows from Theorem~5.2 of \cite{TX} that there is a choice of $K_T$ (see \S\ref{section: GL2} below), a collection of rank $2$ vector bundles $\{\calV_\beta: \beta \in \tilde{T}-T\}$ on $X_{G,K_T}$, and isomorphisms
 $$\xymatrix{X_{T'}  \ar[r]^{\iota_{T'}\ \ \ \ \ \ \ \  }_{\cong\ \ \ \ \ \ \ \ } & \Pi_{_{\beta \in \tilde{T}-T'}} \PP(\calV_{\beta}) \ar[r]^{\ \ \pi_{T'}} &X_{G,K_T},}  $$
where $\PP(\calH_{\beta}) \rightarrow X_{G,K_T}$ is the projectivization of $\calV_\beta$ on $X_{G,K_T}$. 
Furthermore it follows from the proof of \cite[Thm.~5.2]{TX} that there is a collection of line bundles $\{\eta_\beta: \beta \in \BB-\tilde{T}\}$ on $X_{G,K_T}$ such that $\iota_{T'}^*\pi_{T'}^*{\eta_\beta} \cong\omega_\beta$ if $\beta \in \BB-\tilde{T}$, and $ \iota_{T'}^*\calO_{\beta}(-1)\cong\omega_\beta$ for any $\beta \in \tilde{T}-T'$.

 We now define certain morphisms between the strata in $X$. Let $T, T' \subset \BB$ be as above. Then, we have isomorphisms $$\iota_T: X_T \simeq \Pi_{_{\beta \in \tilde{T}-T}} \PP(\calV_{\beta}),$$ 
 $$\iota_{T'}:X_{T'} \simeq \Pi_{_{\beta \in \tilde{T}-T'}} \PP(\calV_{\beta}).$$
Let $\calH^{T'}_\beta=\iota^*_{T'}\pi^*_{T'}(\calV_{\beta})$, a rank $2$ vector bundle on $X_{T'}$. We obtain an isomorphism
\[
\iota_{T,T'}:  X_T \simeq \Pi_{_{\beta \in T'-T}} \PP(\calH^{T'}_\beta).
\]
Let $s_{T'}$ denote the structural morphism $\PP(\calH^{T'}_\beta) \rightarrow X_{T'}$. It follows that 
$\iota^*_{T,T'}s^*_{T'}\omega_\beta\cong \omega_\beta$ for all $\beta \in \BB-T'$, and $\iota^*_{T,T'}\calO_\beta(-1)\cong \omega_\beta$ for $\beta \in T'-T$. We define
$$\pi_{T,T'}:=s_{T'}\circ\iota_{T,T'}: X_T \rightarrow X_{T'}.$$
Then, $\pi^*_{T,T'}\omega_\beta\cong\omega_\beta$ for $\beta \in \BB-T'$, and $\omega_\beta \cong \calO_\beta(-1)$ over $X_T$, for $\beta \in T'-T$.  If $T \subset T'' \subset T'$, then  $\pi_{T,T"}$ and $\pi_{T',T''}$ are defined, and \begin{equation}\label{eqn: cycle}
\pi_{T'',T'} \pi_{T,T''}=\pi_{T,T'}.
\end{equation}

In the following definition, given $T$, we attach various  integers to $\beta \in \BB$ that we will use throughout the article.
\begin{dfn}\label{definition: mu}
Fix $T \subseteq \BB$, and let $\beta \in \BB_\gerp \subseteq \BB$. If $T_\gerp\neq\BB_\gerp$, we define $\mu_\beta$ to be  the smallest integer $i> 0$ such that $\sigma^i\beta \not \in T$. If $T_\gerp=\BB_\gerp$, we define $\mu_\beta=0$.  If $\beta \in \BB_\gerp$, and $\BB_\gerp-\tilde{T}_\gerp \neq \emptyset$, we define $n_\beta$ to be the smallest integer $i>0$ such that $\sigma^i\beta \in \BB_\gerp-\tilde{T}_\gerp$. Finally, if $\beta \in \BB_\gerp \neq T_\gerp$,  we define $\nu_\beta$ to be the smallest integer $i\geq0$ such that $\sigma^{-i}\beta \not \in T_\gerp$. Note that all these notions depend on $T$.\end{dfn}

 \begin{lemma}\label{lemma: pullback} Let $T \subseteq T' \subseteq \BB$ be as above. Assume $\kappa=\sum_{\beta \in \BB}k_\beta e_\beta \in \ZZ^\BB$ be such that $k_{\sigma^i\beta}=0$ for all $\beta \in T'-T, 0\leq i \leq \mu_\beta-1$. Then,  $\pi_{T,T'}^*\omega^\kappa \cong \omega^\kappa$, and the pullback morphism
 $$\pi^*_{T,T'}: H^0(X_{T'}, \omega^\kappa) \rightarrow H^0(X_T, \omega^\kappa)$$
 is an isomorphism.  Furthermore, these isomorphisms are compatible in the sense that if $T''$ is such that $T \subseteq T'' \subseteq T' \subseteq \BB$, then $\pi_{T,T''}^*\circ \pi_{T'',T'}^* = \pi_{T,T'}^*$.
 \end{lemma}
 
\begin{proof} We claim that $\pi^*_{T,T'}\omega_\beta \cong \omega_\beta$, for all $\beta \in J=\BB-\{\sigma^i\tau: \tau \in T'-T, 0\leq i \leq \mu_\tau-1\}$. From the above discussion, we know this holds for all $\tau \in \BB-T'$. Any  $\beta \in J$ is of the form $\sigma^i\beta_0$, where $\beta_0 \in \BB-T'$, and $0 \leq i \leq \mu_{\beta_0}-1$. The claim now follows since for such $\beta_0,i,$ we have an isomorphism $\omega_{\sigma^i\beta_0} \cong \omega_{\beta_0}^{(-p)^i},$ over both $X_T$ and $X_{T'}$. The main statement now  follows since $\pi_{T,T'}: X_T \rightarrow X_{T'}$ is a product of $\PP^1$-bundles.  The compatibility assertion follows from Equation (\ref{eqn: cycle}) and the definition of the  isomorphism $\pi_{T,T'}^*\omega^\kappa \cong \omega^\kappa$.
\end{proof}

\section{Cones of weight for strata}

In the following we will prove the existence of a collection of sections of $\omega^\kappa$ on a stratum $X_T$ (for varying $\kappa$) compatibly with respect to the morphisms described above.

 Let $\gerp|p$, and $\beta,\beta' \in \BB_\gerp$. Choose $0<n\leq |\BB_\gerp|$ such that $\sigma^n\beta'=\beta$. We define $h_\beta^{\beta'}:=-e_{\beta}+ p^n e_{\beta'}$, and $b_\beta^{\beta'}:=e_{\beta}+p^n e_{\beta'}$. When $n=1$, we simply write $h_\beta$, and $b_\beta$, so that we have ${\bf h}_\beta(T) \in H^0(X_T,\omega^{h_\beta})$, for $\beta \not\in T$, and ${\bf b}_\beta(T) \in H^0(X_T,\omega^{b_\beta})$, for $\beta \in T$. We have an identity:
\begin{eqnarray}\label{equation: Hasse} h_{\sigma^{n+m}{\beta}}^{\sigma^n\beta} + p^m h_{\sigma^n\beta}^\beta =h_{\sigma^{n+m}\beta}^{\beta},\end{eqnarray} if $n,m>0$, and $n+m \leq |\BB_\gerp|$.  

For each $T \subseteq \BB$, and  $\gerp|p$, let $\calH(T_\gerp)=\BB_\gerp-T_\gerp$, and $\calT(T_\gerp)=(\BB_\gerp-\tilde{T}_\gerp) \cup \sigma(\tilde{T}_\gerp-T_\gerp)$.

\begin{prop} \label{prop: sections} Let $\beta \in \calH(T_\gerp)$, and $\beta' \in \calT(T_\gerp)$. Write $\beta'=\sigma^{-m}\beta$, where $0<m\leq |\BB_\gerp|$. Then, there exists a nonzero section $${\bf h}_\beta^{\beta'}(T) \in H^0(X_T,\omega^{h_\beta^{\beta'}}),$$ such that 
if $T \subseteq T' \subseteq  T^0:=T \cup \big( \{\sigma^{-j}\beta: 0<j<m\} \cap \tilde{T}\big)$ (whence, $\beta \in \calH(T'_\gerp), \beta'\in \calT(T'_\gerp)$, and $\pi_{T,T'}$ is defined),  then $${\bf h}_{\beta}^{\beta'}(T)=\pi_{T,T'}^*{\bf h}_{\beta}^{\beta'}(T').$$
\end{prop}

\begin{proof} For $\tau \in \BB_\gerp$, let $m_\tau$ be the smallest $j \geq 0$ such that $\sigma^{j}\tau \not \in T^0$, and for $0 \leq i \leq m$,  set $\delta_i=(-1)^{m_{\sigma^{-i}\beta}}$.  Set ${\bf f}_\tau={\bf b}_\tau(T^0)$ if $\tau \in T^0$, and ${\bf f}_\tau={\bf h}_\tau(T^0)$ otherwise. Define $${\bf h}_\beta^{\beta'}(T)=\pi_{T,T^0}^*\displaystyle\prod_{0 \leq i < m} ({\bf f}_{\sigma^{-i}\beta})^{\delta_i p^i},$$
 where $ \prod_{0 \leq i < m} ({\bf f}_{\sigma^{-i}\beta})^{\delta_i p^i}$ is a (nonzero) section of  $\omega^{h_\beta^{\beta'}}$ on $X_{T^0}$, noting that the exponent $\delta_ip^i$ is negative only if $\sigma^{-i}\beta \in T_0$, implying ${\bf f}_{\sigma^{-i}\beta}^{-1}={\bf b}_{\sigma^{-i}\beta}(T^0)^{-1}$  is defined on $X_{T^0}$. By Lemma \ref{lemma: pullback}, it follows that  ${\bf h}_\beta^{\beta'}(T)$ is a nonzero section of $\omega^{h_\beta^{\beta'}}$ on $X_T$. The second statement then follows from the definition of the sections and the compatibility assertion in the lemma.
\end{proof}

\begin{lemma}\label{lemma: intersect} Let $T\subseteq \BB$, and $\gerp |p$. Let $\beta\in \calH(T_\gerp)$, and $\beta' \in \calT(T_\gerp)$. Then the vanishing locus of ${\bf h}_\beta^{\beta'}(T)$ in $X_T$ intersects every connected component of $X_T$ in a nonempty smooth divisor.
\end{lemma}

\begin{proof} First note that if $\iota: Y_1 \rightarrow Y_2$ is a product of $\PP^1$- bundles, and $f$ a section on $Y_2$, the vanishing locus of which intersects every connected component of $Y_2$ in a nonempty smooth divisor, then the vanishing locus of $\iota^*f$ intersects every connected component of $Y_1$ in  a nonempty smooth divisor. Hence, in the notation of Proposition \ref{prop: sections}, it is enough to prove the statement for ${\bf h}_\beta^{\beta'}(T^0)$. This follows from  Lemma 4.1 of \cite{DK1}.  
\end{proof}

Let $ Z$ be a closed subscheme of $X$. We define $C(Z,\omega^\bullet)$ to be the cone in $\ZZ^\BB$ generated by $$\{\kappa \in \BB^\ZZ:  H^0(Z,\omega^\kappa) \neq 0\}.$$
Let $\calC(Z,\omega^\bullet)$ be the $\QQ^{\geq 0}$-saturation of $C(Z,\omega^\bullet)$.
When $Z=X_T$, we denote $C(X_T,\omega^\bullet), \calC(X_T,\omega^\bullet)$ by $C_T, \calC_T$, respectively. We always have
$$\langle \pm b_\beta: \beta \in T; h_\beta: \beta \not\in T \rangle \subset C_T,$$
appearing as weights of $\{{\bf h}_\beta(T): \beta \not\in T\}$ and $\{ {\bf b}_\beta(T)^{\pm 1}: \beta \in T\}$ on $X_T$.
 
 Let $T = \sqcup_{\gerp|p} T_\gerp$, where for each $\gerp|p$, $T_\gerp \subseteq \BB_\gerp$. Let $\tilde{T}=\sqcup_{\gerp|p} \tilde{T}_\gerp$ be as in \S \ref{section: morphisms} 
\begin{dfn}\label{dfn: D cones} Let $\gerp|p$. 
Recall that $\calH(T_\gerp)=\BB_\gerp-T_\gerp$, and $\calT(T_\gerp)=(\BB_\gerp-\tilde{T}_\gerp) \cup \sigma(\tilde{T}_\gerp-T_\gerp)$. We define
$$\calG(T_\gerp)=\{h_\beta^{\beta'}:  \beta \in \calH(T_\gerp), \beta' \in \calT(T_\gerp)\} \cup \{\pm b_\beta: \beta \in T_\gerp \}.$$
We define $D_{T_\gerp}$ to be the $\ZZ^{\geq 0}$-cone generated by $\calG(T_\gerp)$, and $D_T$ the $\ZZ^{\geq 0}$-cone generated by $D_{T_\gerp}$ for all $\gerp|p$. We define $\calD_{T_\gerp}$ and $\calD_T$ to be the  $\QQ^{\geq 0}$-saturation of $D_{T_\gerp}$ and $D_T$, respectively.
\end{dfn}

\begin{cor}\label{cor: lower bound} For any $T \subseteq \BB$, we have $D_T \subseteq C_T$.
\end{cor}
\begin{proof} It is enough to prove that $D_{T_\gerp} \subseteq C_{T}$, for all $\gerp|p$, which follows from Proposition \ref{prop: sections}.
\end{proof}

\section{ Reduced cones and optimal bases}\label{section: optimal}

We  provide an optimal basis for $\calD_{T}$. Fix $\gerp|p$. First assume that $\BB_\gerp-\tilde{T_\gerp}\neq \emptyset$. Recall $n_\beta$ from Definition \ref{definition: mu}. We define 
\begin{eqnarray*}\calG'(T_\gerp) = \{h_{\sigma^{n_\beta}\beta}^\beta: \beta \in \BB_\gerp-\tilde{T_\gerp} \} \cup \{-b_{\sigma^{n_\beta}\beta}^{\beta}: \beta \in \tilde{T_\gerp}-T_\gerp\}  \cup \{\pm b_\beta: \beta \in T_\gerp\}.\end{eqnarray*}
When  $\BB_\gerp-\tilde{T_\gerp}=\emptyset$,  we define
$$\calG'(T_\gerp)= \{-e_{\beta}: \beta \in \BB-T_\gerp\}  \cup \{\pm b_\beta: \beta \in T_\gerp\}.$$
In either case, we define $D'_{T_\gerp}$ to be the $\ZZ^{\geq 0}$-cone generated by $\calG'(T_\gerp)$ in $\ZZ^{\BB_\gerp}$, and $\calD'_{T_\gerp}$ to be the $\QQ^{\geq 0}$-saturation of it. Let $\calD'_{T}$ be the $\QQ^{\geq 0}$-cone generated by $\{\calD'_{T_\gerp}: \gerp|p\}$ in $\QQ^\BB$.

\begin{prop} \label{prop: optimal} We have $\calD'_{T}=\calD_{T}$, for all $T \subseteq  \BB$.
\end{prop}

\begin{proof}  We begin by proving  $\calD'_{T_\gerp}\subseteq \calD_{T_\gerp}$. First, assume $\BB_\gerp-\tilde{T_\gerp}\neq \emptyset$. Clearly, if $\beta \in \BB_\gerp-\tilde{T}_\gerp$, then $h_{\sigma^{n_\beta}\beta}^\beta$ lies in $\calD_{T_\gerp}$. If $\beta \in \tilde{T}_\gerp-T_\gerp$, then 
\[
-b_{\sigma^{n_\beta}\beta}^\beta=h_{\sigma^{n_\beta}\beta}^{\sigma\beta}-p^{(n_\beta-1)}b_{\sigma\beta}.
\]
Since $\sigma^{n_\beta}\beta \in \BB_\gerp-\tilde{T}_\gerp$, and $\sigma\beta \in \sigma(\tilde{T}_\gerp-T_\gerp) \subset T_\gerp$, it follows that both $h_{\sigma^{n_\beta}\beta}^{\sigma\beta}$ and $-b_{\sigma\beta}$ lie in $\calD_{T_\gerp}$, and, hence, so does $-b_{\sigma^{n_\beta}\beta}^\beta$.  If $\BB_\gerp-\tilde{T_\gerp}=\emptyset$, then for $\beta \in \BB_\gerp-T_\gerp=\tilde{T}_\gerp-T_\gerp$, we have $-(1+p^{|\BB_\gerp|})e_\beta=h_\beta^{\sigma\beta}-{p^{|\BB_\gerp|-1}}b_{\sigma\beta} \in D_{T_\gerp}$, and hence $-e_\beta \in \calD_{T_\gerp}$.

We now prove the reverse inclusion. Note that for all $\beta \in \BB_\gerp-\tilde{T}_\gerp$, we have $(p^{|\BB_\gerp|}-1)e_\beta=h_\beta^\beta \in D'_{T_\gerp}$, and hence $e_\beta \in \calD'_{T_\gerp}$. Also, for all $\beta \in \tilde{T}_\gerp-T_\gerp $, we have $-e_\beta \in \calD'_{T_\gerp}$. This is clear if $\tilde{T}_\gerp=\BB_\gerp$. Otherwise, note that $-p^{n_\beta}e_\beta=-b_{\sigma^{n_\beta}\beta}^\beta+e_{\sigma^{n_\beta}\beta} \in \calD_{T'_\gerp}$, as $\sigma^{n_\beta}\beta\in \BB_\gerp-\tilde{T}_\gerp$, whence $e_{\sigma^{n_\beta}\beta} \in \calD'_{T_\gerp}$

To prove $\calD_{T_\gerp}\subset \calD'_{T_\gerp}$, we consider $h_{\beta}^{\beta'}$ in four cases:

\noindent Case 1: $\beta,\beta' \in \BB_\gerp-\tilde{T}_\gerp$. By Identity \ref{equation: Hasse}, we have 
$\langle h_{\sigma^{n_\beta}\beta}^{\beta}: \beta \in \BB_\gerp-\tilde{T_\gerp} \rangle = \langle h_\beta^{\beta'}: \beta,\beta' \in \BB_\gerp-\tilde{T}_\gerp\rangle$. Hence $h_\beta^{\beta'} \in D'_{T_\gerp}$ by definition.

\vspace{2mm}

\noindent Case 2: $\beta \in \BB_\gerp-\tilde{T}_\gerp, \beta' \in \sigma(\tilde{T}_\gerp-T_\gerp)\subseteq T_\gerp$. By Identity \ref{equation: Hasse}, it is enough to assume $\beta=\sigma^{n_{\beta'}}{\beta'}$, in which case, $h_\beta^{\beta'}=-b_\beta^{\sigma^{-1}\beta'} +p^{n_{\beta'}}b_{\beta'} \in \calD'_{T_\gerp}$.

\vspace{2mm}

\noindent Case 3: $\beta \in \tilde{T}_\gerp-T_\gerp, \beta' \in \sigma(\tilde{T}_\gerp-T_\gerp)\subseteq T_\gerp$. Assume $\sigma^n\beta'=\beta$. In this case, we  have $h_\beta^{\beta'}=-e_\beta-p^{n+1}e_{\sigma^{-1}\beta'}+p^nb_{\beta'} \in \calD'_{T_\gerp}$, since $-e_\beta, -e_{\sigma^{-1}\beta'}, b_{\beta'} \in \calD'_{T_\gerp}$.

\vspace{2mm}

\noindent Case 4: $\beta \in \tilde{T}_\gerp-T_\gerp, \beta' \in {\BB}_\gerp-\tilde{T}_\gerp$. Assume $\sigma^n\beta'=\beta$. We have $h_\beta^{\beta'}=-e_{\beta}+ p^ne_{\beta'} \in \calD'_{T_\gerp}$, since $-e_\beta, e_{\beta'} \in \calD'_{T_\gerp}$.
   \end{proof}

\begin{lemma} \label{lemma: no boxes} When $\tilde{T}=T$, we have $\calD_{T}=\langle h_\beta : \beta \in \BB-T \rangle + \langle \pm b_\beta: \beta \in T \rangle.$
\end{lemma}
\begin{proof} By Proposition \ref{prop: optimal}, we  have $\calD_T=\calD'_T=\langle h_{\sigma^{n_\beta}\beta}^\beta: \beta \in \BB-T \rangle  + \langle \pm b_\beta: \beta \in T \rangle$. The result follows since
for  $\beta \not \in T=\tilde{T}$, we have   $h_{\sigma^{n_\beta}\beta}^\beta=h_{\sigma^{n_\beta}\beta}+\Sigma_{i=1}^{n_\beta-1} {(-p)^{n_\beta-i}} b_{\sigma^i\beta}$, noting that each $n_\beta$ is odd in this case.
\end{proof}

For any $\QQ^{\geq 0}$-cone $\langle \pm b_\beta: \beta \in T \rangle \subset C \subset \QQ^{\BB}$, let $C^{T\!-{\rm red}}$ be the image of $C$ under
$$i_T: \QQ^{\BB} \rightarrow \QQ^{\BB-T},$$
which is defined as follows: for $\beta \in \BB_\gerp \subseteq \BB$, we have  $\pi(e_\beta)=0$ if $\BB_\gerp=T_\gerp$, and, otherwise,  $\pi(e_\beta)= (-p)^{\nu_\beta}e_{\sigma^{-\nu_\beta}\beta}$, where $\nu_\beta$ is given in Definition \ref{definition: mu}. Hence, 
\begin{eqnarray}\label{equation: L-beta} i_T (\kappa)=\sum_{\beta \in \BB-T} \big(\sum_{i=0}^{\mu_\beta-1} (-p)^i k_{\sigma^{i}\beta}\big )e_\beta.\end{eqnarray}

The kernel of $\pi$ is the cone generated by all $\{\pm b_\beta: \beta \in T_\gerp\}$, and  the obvious inclusion $j_T: C^{T\!-{\rm red}} \rightarrow C$ gives a section to $\pi$. Hence, we have $C=j_T(C^{T\!-{\rm red}})+ \langle \pm b_\beta: \beta \in T \rangle$. It follows that $D^{{T\!-{\rm red}}}_{T}$ is generated by $\calG^{T\!-{\rm red}}(T)=\cup_{\gerp|p} \calG_\gerp^{T\!-{\rm red}}(T)$, where,
 \begin{eqnarray*} \calG_\gerp^{T\!-{\rm red}}(T)=\{h_{\sigma^{n_\beta}{\beta}}^\beta: \beta \in \BB_\gerp-\tilde{T_\gerp} \} \cup \{-b_{\sigma^{n_\beta}\beta}^{\beta}: \beta \in \tilde{T_\gerp}-T_\gerp\}, \end{eqnarray*}

if $\BB_\gerp-\tilde{T_\gerp}\neq \emptyset$, and  $\calG_\gerp^{T\!-{\rm red}}(T)= \{-e_{\beta}: \beta \in \BB-T_\gerp\},$ if  $\BB_\gerp-\tilde{T_\gerp}=\emptyset$.

 \section {Explicit description of $\calD_T$}\label{section:explicit}
We will now obtain an explicit description of $\calD_T$ in terms of the weight components, i.e., we determine a basis for the dual cone of $\calD_T^{T-{\rm red}}$.  For $\beta \in \BB-\tilde{T}$, we set  $g_\beta=h_{\sigma^{n_\beta}\beta}^{\beta}$, if $\beta \in \BB-\tilde{T}$, and $g_\beta=b_\beta$ if $\beta \in T$.  For $\beta \in \tilde{T}_\gerp-T_\gerp$, we set  $g_\beta=-b_{\sigma^{n_\beta}\beta}^\beta$, if $\BB_\gerp-\tilde{T}_\gerp\neq\emptyset$, and $g_\beta=-e_\beta$, otherwise. By Proposition \ref{prop: optimal}, we have
$$
\calD_T=\langle g_\beta : \beta \in \BB-T \rangle   +  \langle \pm g_\beta: \beta \in T \rangle.$$

Let $\epsilon:\BB \rightarrow \{-1,0,1\}$ be a  map.  Let $T \subseteq \BB$, $\gerp|p$, and $\beta,\beta' \in \BB_\gerp$. Write $\beta'=\sigma^n\beta$ for some  $0 \leq n < |\BB_\gerp|$. define the functional
$$L^\epsilon[\beta,\beta'](\kappa)=\sum_{0 \leq i \leq n} \epsilon(\sigma^i\beta)p^ik_{\sigma^i\beta}.$$
We simply write $L^\epsilon[\beta]$ for $L^\epsilon[\beta,\sigma^{-1}\beta]$.

Given $T \subset \BB$,  we define its sign function $\epsilon_T: \BB \rightarrow \{-1,1\}$. If $\beta \in \BB_\gerp$ and $T_\gerp=\BB_\gerp$, we define $\epsilon_T(\beta)=0$. Otherwise, we define $\epsilon_T(\beta)=1$ for $\beta \in \BB_\gerp-\tilde{T}_\gerp$, and $\epsilon_T(\beta)=(-1)^{\mu_\beta-1}$ for  $\beta \in \tilde{T}_\gerp$. We define for $\beta \in \BB-T$:
\[
L_{T,\beta}=\begin{cases}
L^{\epsilon_T}[\sigma^{\mu_\beta}\beta] & {\rm if}\ \beta \in \BB-\tilde{T},\\
L^{\epsilon_T}[\beta,\sigma^{\mu_\beta-1}\beta]& {\rm if}\ \beta \in \tilde{T}-T.
\end{cases}
\]

\begin{prop} \label{prop:explicit} Let $T \subset \BB$. We have    
$$\calD_T=\{\kappa \in \QQ^\BB: L_{T,\beta}(\kappa) \geq 0,\ \forall \beta \in \BB-T\}.$$

\end{prop}

\begin{proof}   Note that  $\{g_\beta: \beta \in \BB\}$ is a basis for $\QQ^\BB$. By inspection, for $\beta \in \BB-T$, we have 
\[
L_{T,\beta}(g_\tau)=\begin{cases} >0 & {\rm if}\ \tau=\beta\\
=0 & {\rm if}\ \tau\neq \beta.
\end{cases}
\]
Since $\calD_T=\langle g_\beta: \beta \in \BB-T \rangle + \langle \pm g_\beta: \beta \in T\rangle$, it follows that $$\calD_T=\{\kappa \in \QQ^\BB: L_{T,\beta}(\kappa) \geq 0,\ \forall \beta \in \BB-T\}.$$
\end{proof}

\section{The Main Theorem}\label{section: main}

In this section, we determine the cones of weights of strata of $X$.  First, we prove a lemma. Let $Y$ be a scheme over $\FF$, and $\{\calL_i\}_{i\in I}$  a finite collection of line bundles on $Y$. For any weight $\kappa=\sum_{_{i \in I}} \kappa_ie_i \in \ZZ^I$, we define $\calL^{\kappa}=\otimes_{_{i \in I}} \calL_i^{ \kappa_i}$. If $Z$ is a closed subscheme of $Y$, we define  $C_Z(\calL^\bullet)$ to be the $\QQ^{\geq 0}$-cone in $\ZZ^I$ generated by the following subset of $\ZZ^I$:
$$\{\kappa \in \ZZ^I  : H^0(Z, \calL^{\kappa}) \neq 0 \}.$$

\begin{lemma}\label{lemma: intsum}
Let $Y$ be a noetherian scheme over $\FF$, and $\{\calL_i\}_{i\in I}$ a finite collection of line bundles on $Y$.  For $j \in J$ (a finite set), let  ${\bf h}_j \in H^0(Y, \calL^{h_j})$, where $h_j \in  \ZZ^I$.  Let $Z_j$ be the vanishing locus of ${\bf h}_j$, and assume that each $Z_j$ is purely of codimension one. Assume further that the intersection of  $Z_j$ with every irreducible component of $Y$ is nonempty. Then, $$C_Y(\calL^\bullet) \subset \big( \bigcap_{j \in J} C_{Z_j}(\calL^\bullet) \big)+ \langle h_j: j \in J\rangle,$$
where  $\langle h_j: j \in J\rangle$ denotes the $\QQ^{\geq 0}$-cone generated by $\{h_j: j \in J\}$.
\end{lemma}

\begin{proof} Assume $f\neq 0$ is a section of $\calL^\kappa$ on $Y$. By Lemma 7.2 in \cite{DK2}, we can write $f=(\Pi_{_{j \in J}}{\bf h}_j^{M_j}) f_0$ for some nonnegative integers  $\{M_j: j \in J\}$, such that restriction of $f_0$ to each $Z_j$ is nonzero. Therefore, $f_0$ gives a nonzero section of $\calL^{\otimes (\kappa-\sum_{ j\in J}M_jh_j)}$ on each $Z_j$. It follows that  $\kappa-\sum_{j\in J}M_jh_j\in \bigcap_{j \in J} C_{Z_j}(\calL^\bullet)$. 
\end{proof}
 
\begin{remark} Expressions involving cones of weights such as the one appearing in the above lemma arise, and are referred to as intersection-sum cones, in the work of Goldring-Koskivirta on weight cone conjectures. They also implicitly appear in the work of Diamond-Kassaei on minimal cones.
\end{remark}

 \begin{prop} \label{prop: O(1)} Let $T \subseteq \BB$. Let $T \subseteq T'\subseteq   \tilde{T}$ be such that  ${\rm Iw}(T')={\rm Iw}(T)$ (so, in particular, the morphism $\pi_{T,T'}: X_T \rightarrow X_{T'}$ is defined). Let $Z$ be a reduced closed subscheme of $X_{T'}$.  Then
$$\calC(\pi_{T,T'}^{-1}Z ,\omega^\bullet)  \subset \{\kappa=\sum_{_{\beta \in \BB}} k_\beta e_\beta : L_{T,\beta_0}(\kappa) \geq 0,  \forall \beta_0 \in T'-T\}.$$
\end{prop}
\begin{proof} Note that if $T_\gerp=\BB_\gerp$, for some $\gerp|p$, there is nothing to prove at $\beta_0 \in \BB_\gerp$.  Furthermore, by Corollary \ref{cor: torsion}, it is true that for each $\beta \in \BB_\gerp$, the line bundle $\omega_\beta$ is torsion on $X_T$.
Fix $\gerp|p$ such that $T_\gerp \neq  \BB_\gerp$.   For any $\beta \in \BB_\gerp$, define $\nu_\beta$ to be the smallest $i\geq0$ such that $\sigma^{-i}\beta \not \in T_\gerp$. Then, from Definition \ref{dfn: b}, we have $\omega_\beta \cong \omega_{\sigma^{-\nu_\beta}\beta}^{(-p)^{\nu_\beta}}$. For any $\beta \in T_\gerp$, the embedding  $\sigma^{-v_\beta}\beta$ is in $\BB_\gerp-T_\gerp$, so it is either in $\BB_\gerp-\tilde{T}_\gerp$, or in $\tilde{T}_\gerp-T_\gerp$.  Putting all this together, and noting that the subset of $\beta \in \BB_\gerp$ for which $\sigma^{-\nu_\beta}\beta=\beta_0 \in \tilde{T}_\gerp-T_\gerp$ is exactly $\{\beta_0,\sigma\beta_0,...,\sigma^{\mu_{\beta_0}-1}\beta_0\}$, we find that there are integers $\ell_\beta(\kappa)$, for $\beta \in \BB-\tilde{T}$, such that  for some $N>0$, we have an isomorphism of line bundles on $X_T$:
 $$\omega^{N\kappa} \cong \otimes_{_{\beta \in \BB-\tilde{T}}} \omega_\beta^{N\ell_\beta(\kappa)} \otimes \otimes_{_{\beta \in \tilde{T}-T}} \omega_\beta^{NL_{T,\beta}(\kappa)}.$$
By assumptions, we have an isomorphism $\iota_{T,T'}: X_T \rightarrow \Pi_{_{\beta \in I}} \PP(\calH^{T'}_{\beta})$ (see \S \ref{section: morphisms}). Under $\iota_{T,T'}$, this line bundle corresponds to $\calL_N=\otimes_{_{\beta \in \BB-\tilde{T}}} s_{T'}^*\omega_\beta^{N\ell_\beta(\kappa)} \otimes  \otimes_{_{\beta \in \tilde{T}-T'}} s_{T'}^*\omega_\beta^{ NL_{T,\beta}(\kappa)} \otimes \otimes_{_{\beta \in T'-T}} \calO_{\beta}(NL_{T,\beta}(\kappa))$. 

Now assume that $\kappa$ is a weight such that $L_{T,\beta_0}(\kappa)<0,$ for some $\beta_0 \in T'-T$, and let $f \in H^0(\pi_{T,T'}^{-1}Z,\omega^{\kappa})$.
Then, under the isomorphism $i_{T,T'}|_{\pi_{T,T'}^{-1}Z}: \pi_{T,T'}^{-1}Z \rightarrow \Pi_{_{\beta \in T'-T}} \PP(\calH^{T'}_{\beta}|_Z)$, the section  $f^N$ of 
$\omega^{N\kappa}$ on $\pi_{T,T'}^{-1}Z$
gives rise to a section of $\calL_N$ on $\Pi_{_{\beta \in T'-T}} \PP(\calH^{T'}_{\beta}|_Z)$. 
The restriction of $\calL_N$ to every geometric fibre of 
$s_{T'}:\Pi_{_{\beta \in T'-T}} \PP(\calH^{T'}_{\beta}|_Z) \rightarrow Z$ is isomorphic to $\otimes_{_{\beta \in T'-T}} \calO_{\beta}(NL_{T,\beta}(\kappa))$.  
Since $L_{T,\beta_{0}}(\kappa)<0$, the restriction of $f^N$ to any such fibre is zero. Since $Z$ is reduced, so is $\pi_{T,T'}^{-1}Z$, and it follows that $f^N=0$, and hence $f=0$. 
\end{proof}

\begin{lemma}\label{lemma: exceptional} Let $T \subseteq \BB$ be such for some $\gerp|p$, we have $|\BB_\gerp|$ is even, and $T_\gerp=\BB_\gerp-\{\beta\}$, for some $\beta \in \BB_\gerp$. Then $\calC_T \subseteq \{\kappa \in \QQ^\BB: L_{T,\beta}(\kappa)\geq0\}$.
\end{lemma}

\begin{proof}

Note that the assumptions imply $\beta \in \tilde{T}_\gerp-T_\gerp$. This case does not follow from Proposition \ref{prop: O(1)}, since ${\rm Iw}(T)\neq {\rm Iw}(T\cup\{\beta\})$. We prove this by induction on $|\BB-T|$. When $|\BB-T|=1$, we have $T= \BB-\{\beta\}$, and $X_T$ is a union of complete curves. First note that, since $\omega_{\sigma^{-1}\beta} \cong \omega_\beta^{(-p)^{|\BB_\gerp|-1}}$, the Hasse invariant ${\bf h}_\beta(T)$  gives a  section of $\omega_\beta^{-1-p^{|\BB_\gerp|}}$ whose vanishing locus meets every connected component of $X_T$, implying that $\omega_\beta$ is an anti-ample bundle on each such component. Now let  $0 \neq {\bf f} \in H^0(X_T,\omega^\kappa)$. Since $\BB_\gerq \subset T$ for all $\gerq\neq \gerp$, by Corollary \ref{cor: torsion}, we know that $\omega_\tau$ is torsion for all $\beta \in \BB-\BB_\gerp$. Hence, there is $N>0$ such that ${\bf f}^N$ gives a section of $\bigotimes_{\tau \in \BB_\gerp} \omega_\tau^{Nk_\tau}$ on $X_T$. Since all $\beta \neq \tau \in \BB_\gerp$ belong to $T_\gerp$, it follows that $\bigotimes_{\tau \in \BB_\gerp} \omega_\tau^{Nk_\tau} \cong \omega_\beta^{-NL_{T,\beta}(\kappa)}$. Since $\omega_\beta$ is anti-ample on each component, it follows that $L_{T,\beta}(\kappa)\geq 0$. 

Now assume $|\BB-T|>1$, and choose $\beta' \in \BB_\gerq-T_\gerq$, for some $\gerq \neq \gerp$. Lemma \ref{lemma: intsum} implies
$$\calC_T \subset \calC(X_T \cap Z({\bf h}_{\beta'}(T)),\omega^\bullet)+ \langle h_{\beta'}\rangle = \calC_{T \cup \{\beta'\}} + \langle h_{\beta'}\rangle \subset \{\kappa \in \QQ^\BB: L_{T,\beta}(\kappa)\geq 0\},$$ by the induction hypothesis, and since $L_{T,\beta}(h_{\beta'})=0$.
\end{proof}

\begin{cor} \label{cor: general case} Let $T \subseteq \BB$.  Then
$$\calC(X_T ,\omega^\bullet)  \subset \{\kappa=\sum_{_{\beta \in \BB}} k_\beta e_\beta : L_{T,\beta_0}(\kappa) \geq 0,  \forall \beta_0 \in \tilde{T}-T\}.$$
\end{cor}
\begin{proof} Let  $\kappa \in \calC(X_T ,\omega^\bullet)$, and $\gerp|p$. If either $|\BB_\gerp|$ is odd, or $\tilde{T}_\gerp \neq  \BB_\gerp$ it follows from Proposition \ref{prop: O(1)} that $L_{\beta}(\kappa) \geq 0$, for all $\beta \in \tilde{T}_\gerp-T_\gerp$.  Otherwise, assume $\tilde{T}_\gerp=\BB_\gerp$, and $\tilde{T}_\gerp-T_\gerp=\{\beta_1,...,\beta_t\}$. If $t\geq 2$, applying Proposition \ref{prop: O(1)} with $T'=T \cup\{\beta_i\}$, for $1 \leq i \leq t$, implies that $L_{\beta_i}(\kappa) \geq 0$. When $t=1$, Lemma \ref{lemma: exceptional} implies that $L_{\beta_1}(\kappa) \geq 0$.
\end{proof}

We now prove the main result.

\begin{thm} \label{thm:main}For any $T \subset \BB$, we have $$\calC_T=\calD_T=\langle h_{\sigma^{n_\beta}{\beta}}^\beta: \beta \in \BB-\tilde{T} \rangle    + \langle-b_{\sigma^{n_\beta}\beta}^\beta: \beta \in \tilde{T}_\gerp-T_\gerp \rangle + \langle \pm b_\beta: \beta \in T \rangle.$$

\end{thm}

\begin{proof} By Corollary \ref{cor: lower bound} and Proposition \ref{prop: optimal}, it is enough to show $\calC_T \subseteq \calD_T$. We
 prove this by induction on $|\tilde{T}-T|$. The initial case where $\tilde{T}=T$ is already known and can be proved in different ways. In this case, by Proposition \ref{prop: optimal}, and in the notation of \S \ref{section: optimal}, we have 
$$\calD_T=\langle h_{\sigma^{n_\beta}{\beta}}^\beta: \beta \in \BB-T \rangle + \langle \pm b_\beta: \beta \in T \rangle.$$
One can now either use Lemma \ref{lemma: no boxes}, and  appeal to \cite{GK} (where the admissibility condition on a stratum $X_T$ corresponds to the condition $\tilde{T}=T$), or note that the proof of  Corollary 5.4 in \cite{DK1} applies verbatim to show that, in the notation of \S \ref{section: morphisms}, the cone of weights  $\calC(X_{G,K_T},\eta^\bullet)$ is generated by the weights of the  Hasse invariants of $X_{G,K_T}$. It then follows immediately, using $\iota_T: X_T \cong X_{G,K_T}$, that $$\calC_T^{T\!-\rm red}=\iota_T^*\calC(X_{G,K_T},\eta^\bullet).$$ Every Hasse invariant of $X_{G,K_T}$ is a section of a line bundle of the form $(\eta_{\beta})^{p^{n_\beta}}(\eta_{\sigma^{n_\beta}\beta})^{-1}$, where $\beta \in \BB-T$. It pulls back under $\iota_T$ to a nonzero from of weight $h_{\sigma^{n_\beta}\beta}^\beta$.  We deduce that 
$$\calC_T=j_T(\calC_T^{T\!-\rm red})+ \langle \pm b_\beta: \beta \in T\rangle = \langle h_{\sigma^{n_\beta}{\beta}}^\beta: \beta \in \BB-T \rangle + \langle \pm b_\beta: \beta \in T \rangle=\calD_T.$$

We now assume $\tilde{T} \neq T$. Choose $\gerp|p$ such that $\tilde{T}_\gerp-T_\gerp=\{\beta_1,...,\beta_r\}$ for $r>0$. We consider two cases: 

\vspace{2mm}

\noindent {\bf{Case 1:}} $\BB_\gerp-\tilde{T}_\gerp \neq \emptyset$. Without loss of generality, we can assume there is $\beta \in \BB_\gerp-\tilde{T}_\gerp$ such that $\sigma^{\mu_{\beta_1}}\beta_1=\beta$. Let $T''=T\cup\{\beta\}$, $T'=T\cup\{\beta_1\}$. Apply Lemma \ref{lemma: intsum} with $\pi_{T,T'}^*{\bf h}_\beta^\beta(T')={\bf h}_\beta^\beta(T)$ and ${\bf h}_{\beta}(T)$ (using Lemma \ref{lemma: intersect}):
$$\calC_T \subseteq \big(\calC_{T''} \cap \calC(Z(\pi_{T,T'}^*{\bf h}_\beta^\beta(T')),\omega^\bullet) \big)+\langle h_\beta, h_{\beta}^{\beta} \rangle.$$  Applying Proposition \ref{prop: O(1)} with $Z=Z({\bf h}_\beta^\beta(T'))$, we have 
$$\calC_T\subseteq \big(\calC_{T''} \cap \{\kappa: L_{T,\beta_1}(\kappa)\geq 0\}\big)+\langle h_{\beta}, h_{\beta}^{\beta} \rangle.$$  
 By induction hypothesis, $\calC_{T''}=\calD_{T''}$. Since $h_\beta,h_\beta^\beta \in \calD_T$,  to show $\calC_T \subseteq \calD_T$, it is enough to show $$\calD_{T''} \cap \{\kappa: L_{T,\beta_1}(\kappa)\geq 0\} \subset \calD_{T}.$$
Note that we have $\tilde{T''}-T''=(\tilde{T}-T)-\{\beta_1\}$, and $\BB-\tilde{T''}=((\BB-\tilde{T})-\{\beta\})\cup\{\beta_1\}$. We also have $\epsilon_{T''}(\tau)=-\epsilon_T(\tau)$, fot $\tau \in \{\beta_1,\sigma\beta_1,...,\sigma^{\mu_{\beta_1}-1}\beta_1\}$, and  $\epsilon_{T''}(\tau)=\epsilon_T(\tau)$, otherwise. Let $\iota: \BB-\tilde{T} \rightarrow \BB-\tilde{T''}$ be the bijection which is identity except at $\beta$, in which case $\iota(\beta)=\beta_1$. 

Take $\kappa_0 \in \calD_{T''} \cap \{\kappa: L_{T,\beta_1}(\kappa)\geq 0\}$. It follows that  $L_{T,\tau}(\kappa_0)=L_{T'',\tau}(\kappa_0)\geq 0$, for all $\tau \in \tilde{T''}-T''=(\tilde{T}-T)-\{\beta_1\}$, $L_{T,\beta_1}(\kappa_0)\geq 0$, and that, for $\tau \in \BB-\tilde{T}$,  either $L_{T,\tau}(\kappa_0)=L_{T'',\iota(\tau)}(\kappa_0)$ (when $\tau \not \in \BB_\gerp$), or there is $N_\tau\geq 0$ such that $L_{T,\tau}(\kappa_0)=L_{T'',\iota(\tau)}(\kappa_0)+2p^{N_\tau} L_{T,\beta_1}(\kappa_0)\geq 0$.  Proposition \ref{prop:explicit} implies that $\kappa_0 \in \calD_T$.

 \vspace{3 mm}
 
 \noindent {\bf{Case 2:}} $\BB_\gerp-\tilde{T}_\gerp =\emptyset$. Recall that $\tilde{T}_\gerp-T_\gerp=\{\beta_1,\cdots,\beta_r\}$. We have $$\calC_T \subset \{\kappa \in \QQ^\BB: L_{T,\beta_i}(\kappa) \geq 0,  1 \leq i \leq r  \},$$  by Corollary \ref{cor: general case}. Let $T'=T\cup\{\beta_1\}$.
 Therefore, Lemma \ref{lemma: intsum} applied with ${\bf h}_{\beta_1}(T)$   implies that
 $$\calC_T \subseteq \{\kappa \in \calC_{T'} + \langle h_{\beta_1} \rangle:  L_{T,\beta_i}(\kappa)\geq 0, 1 \leq i \leq r\}. $$
 By induction, it is enough to show that $\{\kappa \in \calD_{T'} + \langle h_{\beta_1} \rangle:  L_{T,\beta_i}(\kappa)\geq 0, 1 \leq i \leq r\} \subset \calD_T.$ Since $T,T'$ only differ at $\gerp$, this is equivalent to  $\{\kappa \in \calD_{T'_\gerp} + \langle h_{\beta_1} \rangle:  L_{T,\beta_i}(\kappa)\geq 0, 1 \leq i \leq r\} \subset \calD_{T_\gerp},$ which is clear.

\end{proof}

We now summarize these results and introduce some new notation. Let $T \subseteq \BB$.
First assume $\BB_\gerp-\tilde{T}_\gerp\neq \emptyset$. For all $\beta \in \BB_\gerp-\tilde{T}_\gerp$, we set $f_\beta(T)=h^{\beta}_{\sigma^{n_\beta}{\beta}}$, and ${\bf f}_\beta(T)={\bf h}^{\beta}_{\sigma^{n_\beta}{\beta}}(T)$. For $\beta \in \tilde{T}_\gerp-T_\gerp$, we set $f_\beta(T)=-b^{\beta}_{\sigma^{n_\beta}{\beta}}$, and  ${\bf f}_\beta(T)={\bf h}^{\sigma\beta}_{\sigma^{n_\beta}\beta}(T){\bf b}_{\sigma\beta}(T)^{-p^{n_\beta-1}}$. If $\BB_\gerp-\tilde{T_\gerp}= \emptyset$, for $\beta \in \BB_\gerp-T_\gerp$, we set $f_\beta(T)=-(1+p^{|\BB_\gerp|})e_\beta$, and ${\bf f}_\beta(T)={\bf h}_\beta^{\sigma\beta}(T)({\bf b}_{\sigma\beta}(T))^{-p^{|\BB_\gerp|-1}}$. By Theorem \ref{thm:main}, for any $T \subseteq \BB$, we have 
\[
\calC_T=\langle f_\beta(T): \beta \in \BB-T \rangle + \langle \pm b_\beta: \beta \in T \rangle,
\]
and $0 \neq {\bf f}_\beta(T) \in H^0(X_T,\omega^{f_\beta(T)})$, for all $\beta \in \BB-T$.

Lemma \ref{lemma: intersect} implies that for each $\beta \in \BB-T$, the zero locus of ${\bf f}_\beta(T)$ intersects every connected component of $X_T$ in a smooth divisor. It follows that if $C$ is a connected component of $X_T$, then $$\calC(C,\omega^{\bullet})=\calC(X_T,\omega^{\bullet})=\calC_T.$$

\begin{remark} We defined the cone of weights on $X_T$ in a way that amounts to the sum of the cones of weights of its connected components. An alternative definition would have been to consider the intersection of the cones of weights of the connected components of $X_T$, in which case, Lemma \ref{lemma: intsum} would not generally hold true. However, the two notions coincide for $X_T$ by the discussion above. \end{remark}

\section{Minimal cones} In this section, we determine the minimal cones of strata in the spirit of \cite{DK1,DK2}, and deduce automatic divisibility results. We will first define a stratification on $X_T$, and study its properties which we need in the determination of the minimal cone of $X_T$.

\subsection{A stratification on $X_T$} \label{subsection: stratification} We begin by some preparation. In the rest of this paper, we frequently use the notation introduced at the end of \S \ref{section: main}.

\begin{prop}\label{prop: torsion} Let $T \subseteq \BB$, $\gerp|p$, and $\beta \in \BB_\gerp-T_\gerp$.
\begin{enumerate}
\item If $\tilde{T}_\gerp\neq \BB_\gerp$, then,  on $X_T \cap Z({\bf f}_\beta(T))$, we have $\omega_{\sigma^{n_\beta}\beta}  \cong {\omega_\beta}^{(-p)^{n_\beta}}$, and  on $X_T - Z({\bf f}_\beta(T))$, we have $\omega_{\sigma^{n_\beta}\beta}  \cong {\omega_\beta}^{-(-p)^{n_\beta}}$.
\item If $\tilde{T}_\gerp=\BB_\gerp$, then, on $X_T \cap Z({\bf f}_\beta(T))$, we have $\omega_\beta^{(-p)^{|\BB_\gerp|}-1 }\cong \calO_{X_T\cap Z({\bf f}_\beta(T))}$. We also have   $\omega_\beta^{-(-p)^{|\BB_\gerp|}-1 }\cong \calO_{X_T- Z({\bf f}_\beta(T))}$ on 
 $X_T - Z({\bf f}_\beta(T))$.
\end{enumerate}
\end{prop}
\begin{proof} In case (1), we write $\{\sigma^i\beta: 0 < i < n_\beta\} \cap (\tilde{T}-T)=\{\beta_1,...,\beta_t\}$. Let $T^0=T\cup \{\beta_1,...,\beta_t\}$. The map $\pi_{T,T^0}: X_T \rightarrow X_{T^0}$ restricts to give a morphism $$\pi_{T,T^0}: X_T \cap Z({\bf f}_\beta(T))\rightarrow X_{T^0} \cap Z({\bf f}_\beta(T^0)),$$ 
since, by the second statement in Proposition \ref{prop: sections}, we have  $Z(\pi_{T,T^0}^*{\bf f}_\beta(T^0))=Z({\bf f}_\beta(T))$. In particular, $\pi_{T,T^0}^*\omega_{\tau}=\omega_\tau$ on $X_T \cap Z({\bf f}_\beta(T))$, for $\tau=\beta, \sigma^{n_\beta}\beta$, by the properties of $\pi_{T,T^0}$ discussed in \S \ref{section: morphisms} . The first claim now follows since $\prod_{i=1}^{n_\beta} {\bf b}_{\sigma^i\beta}(T^0\cup\{\sigma^{n_\beta}\beta\})^{-(-p)^{n_\beta-i}}$ is a nowhere vanishing section of $\omega_{\sigma^{n_\beta}\beta}^{-1}\otimes {\omega_\beta}^{(-p)^{n_\beta}}$ on $$X_{T^0} \cap Z({\bf f}_\beta(T^0))=X_{T^0} \cap Z({\bf h}_{\sigma^{n_\beta}\beta}(T^0))=X_{T^0\cup\{\sigma^{n_\beta}\beta\}}.$$ Similarly, we have an induced morphism $$\pi_{T,T^0}: X_T -Z({\bf f}_\beta(T))\rightarrow X_{T^0} - Z({\bf f}_\beta(T^0)),$$ satisfying $\pi_{T,T^0}^*\omega_{\tau}=\omega_\tau$ on $X_T - Z({\bf f}_\beta(T))$, for $\tau=\beta, \sigma^{n_\beta}\beta$, and the second claim follows because  ${\bf h}_{\sigma^{n_\beta}\beta}(T^0)\prod_{i=1}^{n_\beta-1} {\bf b}_{\sigma^i\beta}(T^0)^{(-p)^{n_\beta-i}}$ is a nowhere vanishing section of $\omega_{\sigma^{n_\beta}\beta}^{-1}\otimes {\omega_\beta}^{-(-p)^{n_\beta}}$ on $X_{T^0}-Z({\bf f}_\beta(T^0))=X_{T^0} - Z({\bf h}_{\sigma^{n_\beta}\beta}(T^0))$.

In case (2), letting $T^0=T \cup (\BB_\gerp-\{\beta\})$, and arguing similarly, it is enough to prove the statements over $X_{T^0} \cap Z({\bf f}_\beta(T^0))$. Since  $$X_{T^0} \cap Z({\bf f}_\beta(T^0))=X_{T^0} \cap Z({\bf h}_\beta(T^0))=X_{T^0 \cup \{\beta\}},$$ and $\BB_\gerp \subset T^0 \cup \{\beta\}$, by Corollary \ref{cor: torsion}, we have $\omega_\beta^{(-p)^{|\BB_\gerp|}-1}\cong\calO_{X_{T^0} \cap Z({\bf f}_\beta(T^0)}$. The second claim follows  since   ${\bf h}_{\beta}(T^0)\prod_{i=1}^{n_\beta-1} {\bf b}_{\sigma^i\beta}(T^0)^{(-p)^{n_\beta-i}}$  is a nowhere vanishing section of $\omega_\beta^{-(-p)^{|\BB_\gerp|}-1}$ on $X_{T^0} - Z({\bf f}_\beta(T^0))=X_{T^0}  -Z({\bf h}_\beta(T^0))$.
 
\end{proof}

\begin{dfn} Let $T \subseteq \BB$, and $J \subseteq \BB-T$. We define $$X_{T,J}=X_T \cap \big(\bigcap_{\beta \in J} Z({\bf f}_\beta(T))\big),$$
$$W_{T,J}=X_{T,J}-{\bigcup}_{\beta \not \in J} X_{T,\{\beta\}}.$$ \end{dfn}

\begin{prop}\label{prop: quasi-affine} If $T \subseteq \BB$ and $J \subseteq \BB-T$, then $W_{T,J}$ is quasi-affine.
\end{prop}
\begin{proof} Using all the isomorphisms given in Proposition \ref{prop: torsion} and Definition \ref{dfn: b}, it is easy to see that  for all $\beta \in \BB$, $\omega_\beta$ is a torsion line bundle over $W_{T,J}$. Since $\bigotimes_{\beta \in \BB} \omega_\beta$ is ample on $X$, it follows that its restriction to $W_{T,J}$ is both ample and torsion, which implies that $W_{T,J}$ is quasi-affine (since it is quasi-compact).
\end{proof}
  
  \begin{cor}\label{cor: intersection} Let $T \subset \BB$, $J \subset \BB-T$, and assume $X_{T,J}$ is nonempty.  Then, for any $\beta \not\in T\cup J$, $X_{T,\{\beta\}}$ intersects every connected component of $X_{T,J}$.
 \end{cor}
 
 \begin{proof} The case $T=\emptyset$ has been proved in Corollary 5.7 of \cite{DK2}. We assume $T \neq \emptyset$, and, therefore, $X_T$ is proper. Let $J\subset \BB-T$ be a subset of maximal cardinality for which the claim fails to be true (in particular $J \neq \BB-T$). Let $C$ be a connected component of $X_{T,J}$ such that $C \cap X_{T,\{\beta\}}=\emptyset$ for some $\beta \not \in T\cup J$. The maximality assumption implies that $C \cap X_{T,\{\tau\}}=\emptyset$, for all $\tau \not\in T\cup J$. It follows that $C \subset W_{T,J}$. Since $W_{T,J}$ is quasi-compact, Proposition \ref{prop: quasi-affine} implies that $C$ is quasi-affine. Since $C$ is also proper, it follows that $C$ is finite. This is a contradiction, since $X_{T,J}$ is cut out from $X_T$ by the vanishing of $|J|$ sections, and, hence, every connected component of $X_T$ has codimension at most  $|J| < |\BB-T|={\rm dim}\ \!  X_T$.
 
 \end{proof}
 
    \begin{prop}  \label{prop:strata} Let $T \subseteq \BB$, $J \subseteq \BB-T$. We have
 
\begin{enumerate}
\item $X_{T,J}=\overline{W}_{T,J}=\bigcup_{J' \supseteq J} W_{T,J'}$.
\item Each $W_{T,J}$ (and hence $X_{T,J}$) is non-empty and equidimensional of dimension $|\BB-(T \cup J)|$.
\
\end{enumerate} \end{prop}

\begin{proof} This can be proved exactly as in Proposition 5.8 of \cite{DK2} using proposition \ref{prop: quasi-affine}, and Corollary \ref{cor: intersection}.
\end{proof}

\begin{remark} The above construction gives a stratification on 
(the connected components of) $X_T$ which is different from the one induced by the Goren-Oort stratification on $X$ in general.
\end{remark}

\subsection{The minimal cone of $X_T$} In this section, we define and determine the minimal cone of $X_T$. We begin by stating some consequences of the results in the previous section that we will need in our arguments.

 \begin{lemma}\label{lemma: h intersect f} Let $T \subseteq \BB$, and $J \subseteq \BB-T$. Let $\beta \in \BB-(T \cup J)$, and $T'=T \cup\{\beta\}$. Then $X_{T',J}$ intersects every connected component of $X_{T,J}$.
\end{lemma}

\begin{proof} Let $C$ be a connected component of $X_{T,J}$. By Corollary \ref{cor: intersection} and Proposition \ref{prop:strata}, we know that $Z:=C \cap X_{\BB-(J \cup \{\beta\}),J}$ is nonempty, projective, and of pure dimension $1$. Assume $X_{T',J}\cap C =\emptyset$. Then $Z \subseteq  X_{T,J}-X_{T,J'}$. In particular, $Z \cap X_{T,J'}\neq \emptyset$, which implies that $\omega_\beta\cong\omega_{\sigma^{-1}\beta}^{-p}$ on $Z$. This, together with all the isomorphisms obtained via Proposition \ref{prop: torsion}, imply that $\omega_\tau$ is torsion on $Z$, for all $\tau \in \BB$. In particular,  we find that $\bigotimes_{\tau\in \BB} \omega_\tau$ is both ample and torsion on $Z$. This implies that $Z$ is quasi-affine, which is a contradiction. \end{proof}

\begin{lemma}\label{lemma: survive} Let $T \subseteq \BB$, $\beta \in \BB-T$, and $C$ a connected component of $X_T \cap Z({\bf f}_\beta(T)$. We have $$\langle f_\tau(T): \tau \in \BB-(T \cup \{\beta\}) \rangle + \langle \pm b_\tau: \tau \in T \rangle \subset  \calC(C, \omega^\bullet).$$In particular, $\langle f_\tau(T): \tau \in \BB-(T \cup \{\beta\}) \rangle + \langle \pm b_\tau: \tau \in T \rangle \subset  \calC(X_T \cap Z({\bf f}_\beta(T)), \omega^\bullet)$.
\end{lemma}
\begin{proof} Clearly, for all $\tau \in T$, we have $\pm b_\tau \in \calC(X_T \cap Z({\bf f}_\beta(T)), \omega^\bullet)$. If $\tau  \in \BB-(T \cup \{\beta\})$, then ${\bf f}_\tau(T)$ restricts to a nonzero section of $\omega^{f_\tau(T)}$ on any connected component of $X_T \cap Z({\bf f}_\beta(T))=X_{T,\{\beta\}}$, by Lemma \ref{lemma: h intersect f}.
\end{proof}

 Given $T \subset \BB$, and $\gerp|p$, we define  $\BB_\gerp(T)$ as follows:
 $$\BB_\gerp(T)=\begin{cases} \BB_\gerp-T_\gerp & {\rm if}\ |\BB_\gerp-\tilde{T}_\gerp| >1, \\
  \BB_\gerp-(T_\gerp \cup \{\beta\}) & {\rm if}\ \BB_\gerp-\tilde{T}_\gerp=\{\beta\}, \\
  \emptyset & {\rm if}\ \BB_\gerp-\tilde{T}_\gerp=\emptyset.
 \end{cases}$$ 
We define $\BB(T)=\bigcup_{\gerp|p} \BB_\gerp(T)$. Note that $\beta \in \BB(T)$ is equivalent to the condition that 
$n_\beta$ is defined, and $\sigma^{n_\beta}\beta \neq \beta$.

\begin{dfn} For any $\beta \in \BB(T)$, define $\epsilon_{T,f_\beta}: \BB \rightarrow  \{-1,0,1\}$ as
$$\epsilon_{T,f_\beta}(\tau)=\begin{cases} -\epsilon_T(\tau) & {\rm if}\ \tau \in \{\beta,\sigma\beta,...,\sigma^{\mu_\beta-1}\beta\}, \\ \ \epsilon_T(\tau) &{\rm otherwise.}
 \end{cases}$$ 
 Given $\beta \in \BB(T)$, using notation from \S \ref{section:explicit}, for  $\tau \in \BB-(T\cup \{\sigma^{n_\beta}\beta\})$, we define 
 $$L_{T,f_\beta,\tau}=\begin{cases} L^{\epsilon_{T,f_\beta}}[\sigma^{\mu_\tau}\tau]& {\rm if}\  \beta \neq \tau \in \BB-\tilde{T}, \\  
 L_{T,\tau}& {\rm if}\  \beta \neq \tau \in \tilde{T}-T,\\
 L^{\epsilon_{T,f_\beta}}[\beta,\sigma^{-1}\tilde{\beta}]&{\rm if}\ \beta=\tau \in \BB-\tilde{T}, \\  
 L^{\epsilon_{T,f_\beta}}[\tilde\beta]&{\rm if}\  \beta=\tau \in \tilde{T}-T,
 \end{cases}$$
 where, $\beta'={\sigma^{n_\beta}\beta}$, and $\tilde{\beta}=\sigma^{\mu_{\beta'}}\beta'$. Note that if $\beta \in \BB_\gerp$, then, $L_{T,f_\beta,\tau}=L_{T,\tau}$, for all $\tau \not \in \BB_\gerp$. Finally, for $\beta \in \BB(T)$, we define $$\calD_{T,f_\beta}=\{\kappa \in \QQ^\BB: L_{T,f_\beta,\tau}(\kappa)\geq 0: \forall \tau \in \BB-(T\cup \{\sigma^{n_\beta}\beta\})\}.$$
 
 \end{dfn}

   \begin{thm} \label{thm: general cones} Let $T \subset \BB$, and $\beta \in \BB-T$. \begin{enumerate}
\item If $\beta \in\BB(T)$, we have 
\begin{enumerate}
\item ${\rm dim}\ H^0(C,\omega^{f_\beta(T)})=1$, for any connected component $C$ of $X_T$.
\item $\calC(X_T \cap Z({\bf f}_\beta(T)),\omega^\bullet)=\calD_{T,f_\beta}.$
\end{enumerate}
\item If $\beta\not\in \BB(T)$, we have $\calC_T  \subseteq \calC(X_T \cap Z({\bf f}_\beta(T)),\omega^\bullet)$.
\end{enumerate}
\end{thm}

\begin{proof} First, we show that for $\beta \in \BB(T)$,  (1)(b)  implies (1)(a). Let $C$ be a connected component of $X_T$, and ${\bf f} \in H^0(C,\omega^{f_\beta(T)})$. Let $\tilde{\bf f}$ denote the extension of ${\bf f}$ to $X_T$ by setting it zero on the rest of the connected components. Since $\beta \in \BB(T)$,  we can easily see that there is $\delta \geq 0$ such that $L_{T,f_\beta,\beta}(f_\beta(T))=-2p^{n_\beta+\delta}<0$, which implies that $f_\beta(T) \not \in \calD_{T,f_\beta}$. By  (1)(b), it follows that $f_\beta(T) \not \in \calC(X_T\cap Z({\bf f}_\beta(T)))$, and so the restriction of $\tilde{\bf f}$ to $X_T \cap Z({\bf f}_\beta(T))$ is zero. This means $\bf f$ is divisible by the restriction of ${\bf f}_\beta(T)$ to $C$ (which is nonzero by the discussion at the end of \S \ref{section: main}), from which the claim follows. 

Next, we prove (1)(b). We begin by show that  $\calD_{T,f_\beta} = \calC(X_T\cap Z({\bf f}_\beta(T)),\omega^\bullet)$, for all $\beta \in \BB(T)$, if we have   $\{\sigma^i\beta: 0 < i <n_\beta\} \cap (\tilde{T}-T)=\emptyset$. In this case, Proposition \ref{prop: sections} tells us that  ${\bf f}_\beta(T)={\bf h}_{\sigma^{n_\beta}\beta}(T)\displaystyle\prod_{0 < i < n_\beta} {\bf b}_{\sigma^{i}\beta}(T)^{(-p)^{i}}$, since, by assumption,  $\{\sigma^i\beta: 0 \leq i < n_\beta\}\subset T$, and $T^0=T$ (in the notation of the proposition). It follows that $X_T \cap Z({\bf f}_\beta(T))=X_T \cap Z({\bf h}_{\sigma^{n_\beta}\beta}(T))=X_{T \cup \{\sigma^{n_\beta}\beta\}}$, from which the claim follows by Proposition \ref{prop:explicit} and Theorem \ref{thm:main}, noting that in this case, $L_{T,f_\beta,\tau}=L_{T\cup \{\sigma^{n_\beta}\beta\},\tau}$, for all $\tau \in \BB-(T\cup\{\sigma^{n_\beta}\beta\})$.  

We now consider a general $T$, and prove  $\calD_{T,f_\beta} \subset \calC(X_T\cap Z({\bf f}_\beta(T)),\omega^\bullet)$, for all $\beta \in \BB(T)$, using an  induction on the cardinality of $\{\sigma^i\beta: 0 < i <n_\beta\} \cap (\tilde{T}-T)$. The base case was proved above.  Now assume $\{\sigma^i\beta: 0 < i < n_\beta\} \cap (\tilde{T}-T)=\{\sigma^{r_i}\beta: 1 \leq i \leq t\}$, where $0 < r_1 <...<r_t<n_\beta$, and $t\geq 1$.   Let $\kappa \in \calD_{T,f_\beta}$. By discussions in \S \ref{section: optimal} (See Equation \ref{equation: L-beta}), for $r_t<i<n_\beta$, we can find $\lambda_{\sigma^i\beta} \in \QQ$ such that 
\[
\kappa':=\kappa + \sum_{r_t <i < n_\beta} \lambda_{\sigma^i\beta} b_{\sigma^i\beta} -p^{r_t-n_\beta}L_{T,\sigma^{r_t}\beta}(\kappa)  f_{\sigma^{r_t}\beta}(T),
\]
satisfies $k'_{\tau}=0$, if $\tau=\sigma^i\beta$ for $r_t \leq i < n_\beta$, and $k'_{\sigma^{n_\beta}\beta}=k_{\sigma^{n_\beta}\beta}+p^{r_t-n_\beta}L_{T,\sigma^{r_t}\beta}(\kappa)$.

Since $\kappa \in \calD_{T,f_\beta}$, we have  $L_{T,\sigma^{r_t}\beta}(\kappa)=L_{T,f_\beta,\sigma^{r_t}\beta}(\kappa)\geq 0$.  Lemma \ref{lemma: survive} now implies that $\kappa \in \calC(X_T\cap Z({\bf f}_\beta(T)),\omega^\bullet)$ follows from  $\kappa' \in \calC(X_T\cap Z({\bf f}_\beta(T)),\omega^\bullet)$.   We let $T'=T \cup \{\sigma^{r_t}\beta\}$, and note that $f_\beta(T')=f_\beta(T)$, meaning we have   $0 \neq {\bf f}_\beta(T') \in H^0(X_{T'},\omega^{f_\beta(T)})$. Furthermore, by the second statement in Proposition \ref{prop: sections}, we have $X_T \cap Z(\pi_{T,T'}^*{\bf f}_\beta(T'))=X_T \cap Z({\bf f}_\beta(T))$. Therefore,  $\pi_{T,T'}: X_T \rightarrow X_{T'}$ (defined in \S \ref{section: morphisms}) restricts to give a $\PP^1$-bundle morphism
\[
\pi_{T,T'}: X_T \cap Z({\bf f}_\beta(T)) \rightarrow X_{T'} \cap Z({\bf f}_\beta(T')).
\]
Hence, Lemma \ref{lemma: pullback} implies that
\[
\pi^*_{T,T'}: H^0(X_{T'}\cap Z({\bf f}_\beta(T')), \omega^{\kappa'}) \rightarrow H^0(X_T\cap Z({\bf f}_\beta(T)) , \omega^{\kappa'})
\]
is an isomorphism (noting that $\mu_{\sigma^{r_t}\beta}=n_\beta-r_t$). Consequently,   it is enough to prove that $\kappa' \in \calD_{T',f_\beta} \subseteq \calC(X_{T'}\cap Z({\bf f}_\beta(T')),\omega^\bullet)$, where the inclusion holds  by the induction hypothesis. This follows by checking that for all $\tau \in \BB-(T \cup \{\sigma^{r_t}\beta,\sigma^{n_\beta}\beta\})$, we have $L_{T',f_\beta,\tau}(\kappa')=L_{T,f_\beta,\tau}(\kappa)$ (by considering the cases $t=1$ and $t>1$ separately).

To prove the reverse inclusion, i.e.,  $\calC(X_T\cap Z({\bf f}_\beta(T)),\omega^\bullet)\subset \calD_{T,f_\beta}$, for $\beta \in \BB(T)$, we again use an induction on the cardinality of $\{\sigma^i\beta: 0 < i <n_\beta\} \cap (\tilde{T}-T)=\{\sigma^{r_i}\beta: 1 \leq i \leq t\}$ (the base case having already been proved above). Let $T'=T \cup \{\sigma^{r_t}\beta\}$, and $T^{''}=T \cup \{\sigma^{r_1}\beta,...,\sigma^{r_t}\beta \}$. Note that $f_\beta(T'')=f_\beta(T)$, and $f_{\sigma^{n_\beta}\beta}(T'')=f_{\sigma^{n_\beta}\beta}(T)$.  Considering $\pi_{T,T''}: X_T \rightarrow X_{T''}$, and using an argument as in the case of $\pi_{T,T'}$ above, we see that we have a $(\PP^1)^t$-bundle map $$\pi_{T,T''}: Y:=X_T \cap Z({\bf f}_\beta(T)) \rightarrow Y'':=X_{T''} \cap Z({\bf f}_\beta(T'')).$$

We want to apply Lemma \ref{lemma: intsum} on $X_T \cap Z({\bf f}_\beta(T))$ with ${\bf h}_{\sigma^{r_t}\beta}(T)$ and $\pi_{T,T''}^*{\bf f}_{\sigma^{n_\beta}\beta}(T'')={\bf f}_{\sigma^{n_\beta}\beta}(T)$. These sections satisfy the conditions required in Lemma \ref{lemma: intsum}, by Lemma \ref{lemma: h intersect f} (and noting that  $\sigma^{n_\beta}\beta \neq \beta$, $\sigma^{r_t}\beta \neq \beta$, $\sigma^{n_\beta}\beta\not\in T$). By Lemma \ref{lemma: intsum}, we have 
$$\calC(Y,\omega^\bullet) \subseteq \big(\calC(Y\cap Z({\bf h}_{\sigma^{r_t}\beta}(T)),\omega^\bullet) \cap \calC(Y \cap Z(\pi_{T,T''}^*{\bf f}_{\sigma^{n_\beta}\beta}(T'')),\omega^\bullet) \big)+\langle h_{\sigma^{r_t}\beta}, f_{\sigma^{n_\beta}\beta}(T) \rangle.$$

 Note that $Y\cap Z({\bf h}_{\sigma^{r_t}\beta}(T))=\big(X_T \cap Z({\bf f}_\beta(T)\big) \cap Z({\bf h}_{\sigma^{r_t}\beta}(T))$ can be thought of as the vanishing locus of ${\bf f}$, the restriction of ${\bf f}_\beta(T)$ to $X_T \cap Z({\bf h}_{\sigma^{r_t}\beta}(T))=X_{T'}$. 
  The section ${\bf f}$ is a  section of $\omega^{f_\beta(T)}=\omega^{f_\beta(T')}$. We prove it is nonzero on every connected component  of $X_{T'}$. Since, by discussions in \S\ref{section: morphisms}, $X_T$ is a $\PP^1$-bundle over $X_{T'}$, the connected components of $X_T$ and $X_{T'}$ are in bijection. By Lemma 4.1 of \cite{DK1}, the connected components of $X_{T'}$ map onto those of $X_T$ under the natural embedding $X_T' \rightarrow X_T$.  Hence, given a connected component $C'$ of $X_{T'}$, there is a connected component $C$ of $X_T$ such that $C'=C \cap X_{T'}=C \cap Z({\bf h}_{\sigma^{r_t}\beta}(T))$. If ${\bf f}$ is zero on $C'$, it follows that ${\bf f}_\beta(T)$ is divisible by ${\bf h}_{\sigma^{r_t}\beta}(T)$ over $C$. It follows that $\lambda:=f_\beta(T)-h_{\sigma^{r_t}\beta}\in \calC(C,\omega^\bullet)$, being the weight of the nonzero section ${\bf f}_\beta(T)/{\bf h}_{\sigma^{r_t}\beta}(T)$ on $C$.  Since  $L_{T,\sigma^{r_t}\beta}(\lambda)=-1<0$,  Proposition \ref{prop:explicit} implies that $\lambda \not\in \calD_T=\calC_T=\calC(C, \omega^{\bullet}).$ (For the last equality, see discussions at the end of \S \ref{section: main}). This is a contradiction. 

 Therefore,  $\bf f$ is nonzero on every connected comopnent of $X_{T'}$, and by (1)(a) for $T'$ (which follows from  (1)(b) for $T'$, by the induction hypothesis) $\bf f$ must be a nonzero scalar multiple of ${\bf f}_\beta(T')$ on each connected component of $X_T'$. We deduce that  $$Y\cap Z({\bf h}_{\sigma^{r_t}\beta}(T))=X_{T'} \cap Z({\bf f}_\beta(T')).$$
  By Proposition \ref{prop: O(1)}, we have 
$$\calC(Y \cap Z(\pi_{T,T'}^*{\bf f}_{\sigma^{n_\beta}\beta}(T'')),\omega^\bullet) \subseteq  \{\kappa: L_{T,\sigma^{r_i}\beta}(\kappa)\geq 0,\ \forall \ 1 \leq i \leq t\},$$  
and, by the induction hypothesis, we have  $$\calC(Y\cap Z({\bf h}_{\sigma^{r_t}\beta}(T)),\omega^\bullet)=\calC(X_{T'} \cap Z({\bf f}_\beta(T')),\omega^\bullet) \subseteq\calD_{T',f_\beta}.$$ Putting all these together, we have 
$$\calC(Y,\omega^\bullet) \subseteq  \big( \calD_{T',f_\beta} \cap   \{\kappa: L_{T,\sigma^{r_i}\beta}(\kappa)\geq 0,\ \forall \ 1 \leq i \leq t\}\big) +\langle h_{\sigma^{r_t}\beta}, f_{\sigma^{n_\beta}\beta}(T) \rangle.$$ 
It can be checked that $h_{\sigma^{r_t}\beta}, f_{\sigma^{n_\beta}\beta}(T) \in \calD_{T,f_\beta}$. Therefore, it is enough to show $$\calD_{T',f_\beta} \cap   \{\kappa: L_{T,\sigma^{r_i}\beta}(\kappa)\geq 0,\ \forall \ 1 \leq i \leq t\}\big)  \subset \calD_{T,f_\beta}.$$
First assume $t>1$, in which case $\epsilon_{T,f_\beta}=\epsilon_{T',f_\beta}$, and from which it  is easy to check that $L_{T,f_\beta,\tau}(\kappa)=L_{T',f_\beta,\tau}(\kappa)\geq 0$ for all $\tau \in \BB-(T\cup\{\sigma^{n_\beta}\beta\})$ such that $\tau \not \in \{\sigma^{r_t}\beta, \sigma^{r_{t-1}}\beta\}$. The result follows since the two remaining conditions for $\kappa\in \calD_{T,f_\beta}$ are also satisfied: $L_{T,f_\beta,\tau}(\kappa)=L_{T,\tau}(\kappa)\geq 0$ for $\tau \in \{\sigma^{r_t}\beta, \sigma^{r_{t-1}}\beta\}$. Now assume $t=1$. In this case, $\epsilon_{T,f_\beta}(\tau)=-\epsilon_{T',f_\beta}(\tau)$ if $\tau \in \{\sigma^{r_1}\beta,...,\sigma^{n_\beta-1}\beta\}$, and $\epsilon_{T,f_\beta}(\tau)=\epsilon_{T',f_\beta}(\tau)$ otherwise. An examination of the defintion shows that for all $\sigma^{r_1}\beta \neq \tau  \in \BB-(T\cup\{\sigma^{n_\beta}\beta\})$, either $L_{T,f_\beta,\tau}=L_{T',f_\beta,\tau}$, or there is $N_\tau \geq 0$, such that  $L_{T,f_\beta,\tau}-L_{T',f_\beta,\tau}=2p^{N_\tau} L_{\sigma^{r_1}\beta}(\kappa)\geq 0$. The last condition for $\kappa\in \calD_{T,f_\beta}$ is $L_{T,f_\beta,\sigma^{r_1}\beta}(\kappa)=L_{T,\sigma^{r_1}\beta}(\kappa) \geq 0$, which also holds by the assumption.

We now prove (2). By assumption, either $\BB_\gerp-\tilde{T}_\gerp=\emptyset$ (and $\beta \in \tilde{T}_\gerp-T_\gerp$), or $\BB_\gerp-\tilde{T}_\gerp=\{\beta\}$. By Lemma \ref{lemma: survive}, it is enough to show that $f_\beta(T) \in \calC(X_T \cap Z({\bf f}_\beta(T)),\omega^\bullet)$.   Let $T'=T \cup (\BB_\gerp-\{\beta\})$, so that, by Proposition \ref{prop: sections}, we have $$\pi_{T,T'}: X_T \cap Z({\bf f}_\beta(T)) \rightarrow X_{T'} \cap Z({\bf f}_\beta(T'))$$
is the restriction of $\pi_{T,T'}:X_T \rightarrow X_{T'}$ to $X_{T'} \cap Z({\bf f}_\beta(T'))$. By Lemma \ref{lemma: pullback}, it is enough to show that $f_\beta(T) \in \calC(X_{T'}\cap Z({\bf f}_\beta(T')),\omega^\bullet)$. But, by Proposition \ref{prop: sections}, we have   $X_{T'}\cap Z({\bf f}_\beta(T'))=X_{T'}\cap Z({\bf h}_\beta(T'))=X_{T'\cup\{\beta\}}$, and $\Sigma_{\tau \in \BB_\gerp} \QQ e_\tau \subseteq \calC_{T'\cup\{\beta\}}$, as $\BB_\gerp \subseteq T'\cup\{\beta\}$.
 \end{proof}
 
 \begin{corollary}\label{cor: divisibility}Let $T \subseteq \BB$. Assume ${\bf f}\in H^0(X_T,\omega^\kappa)$, and  $L_{T,f_\beta,\tau}(\kappa)<0$, for some $\beta \in \BB(T)$, $\tau \in \BB-(T \cup\{\sigma^{n_\beta}\beta\})$. Then ${\bf f}$ is divisible by ${\bf f}_\beta(T)$.
\end{corollary}

\begin{proof} By Theorem \ref{thm: general cones}, the assumption implies that $\kappa \not \in \calC(X_T \cap Z({\bf f}_\beta(T)), \omega^\bullet)$. Hence, the restriction of $\bf f$ to $X_T \cap Z({\bf f}_\beta(T))$ is zero, which implies that ${\bf f}$ is divisible by ${\bf f}_\beta(T)$.
\end{proof}

\begin{dfn} Let $T \subseteq \BB$. The minimal cone of $X_T$ is defined to be 
$$\calC^{\rm min}_T=  i_T( \calC_T \cap (\bigcap_{\beta \in \BB(T)}  \calD_{T,f_\beta}))\subseteq \calC_T^{T-{\rm red}} \subseteq \QQ^{\BB-T},$$
where $i_T:\QQ^{\BB} \rightarrow \QQ^{\BB-T}$ is defined in \S \ref{section: optimal}). Note that for all $\beta \in T$, we have $${\rm Ker} (i_T)=\langle \pm b_\beta: \beta \in T \rangle \subseteq  \calC_T \cap (\bigcap_{\beta \in \BB(T)}  \calD_{T,f_\beta}).$$  
It follows that $\calC_T \cap (\bigcap_{\beta \in \BB(T)}  \calD_{T,f_\beta})= i_T^{-1}(\calC^{\rm min}_T)=j_T(C_T^{\rm min}) + \langle \pm b_\beta: \beta \in T \rangle$.

\end{dfn}

\begin{prop} \label{prop: diagonal} Let $T \subseteq \BB$ be such that $\tilde{T}=T$. Then $$\calC^{\rm min}_T=\{\sum_\beta \ell_\beta e_\beta \in \QQ^{\BB-T}: p^{n_\beta} \ell_{\sigma^{n_\beta}\beta} \geq \ell_\beta, \forall \beta \in \BB-T\}.$$
 \end{prop}
 
 \begin{proof} The projection $i_T: \QQ^\BB \rightarrow \QQ^{\BB-\TT}$ is given by $i_T(\sum_{\beta\in\BB} k_\beta e_\beta)=\sum_{\beta \in \BB-T} \ell_\beta e_\beta$, where, for $\beta \in \BB-T$, we have $\ell_\beta=\sum_{i=0}^{\mu_\beta-1} (-p)^i k_{\sigma^i\beta}$. Note that  for $\beta \in \BB-\tilde{T}=\BB-T$,  we have $$L_{T,f_\beta,\beta}(\sum_{\beta\in\BB} k_\beta e_\beta )=p^{n_\beta} \ell_{\sigma^{n_\beta}\beta} - \ell_\beta.$$ This proves one inclusion.  For the reverse inclusion, we need to show that if $\kappa=\sum_{\beta\in\BB} k_\beta e_\beta$ is such that $L_{T,f_\beta,\beta}(\sum_{\beta\in\BB} k_\beta e_\beta )=p^{n_\beta} \ell_{\sigma^{n_\beta}\beta} - \ell_\beta \geq 0$, for all $\beta \in \BB-T$, then $\kappa \in \calC_T \cap (\bigcap_{\beta \in \BB-T}  \calD_{T,f_\beta})$. Note that these conditions imply $\ell_\beta \geq 0$ for all $\beta \in \BB-T$. For $\gerp|p$, and  $\tau \in \BB_\gerp-T_\gerp$, set $\bar{\sigma}\tau:=\sigma^{n_\tau}\tau\in \BB_\gerp-T_\gerp$, and for $1 \leq i \leq |\BB_\gerp-T_\gerp|$, let  $0 < n_i(\tau) \leq |\BB_\gerp|$ be such that $\bar{\sigma}^i\tau=\sigma^{n_i(\tau)}\tau$. Then,   $$p^{n_1(\tau)}L_{T,\tau}(\kappa)=\sum_{i=1}^{|\BB_\gerp-T_\gerp|} p^{n_i(\tau)}\ell_{\bar{\sigma}^i\tau} \geq 0.$$
 It remains to show that $L_{T,f_\beta, \tau}(\kappa)\geq0$, for all $\beta \in \BB(T)=\BB-T$ and $\beta \neq \tau \in \BB-(T\cup\{\sigma^{n_\beta}\beta\})$. If $\beta \in \BB_\gerp$ and $\tau\not\in \BB_\gerp$, we have  $L_{T,f_\beta, \tau}(\kappa)=L_{T,\tau}(\kappa) \geq 0$ from above. Hence, we can assume $\tau \in \BB_\gerp$, and pick $1\leq j \leq |B_\gerp-T_\gerp|$ such that $\bar{\sigma}^j\tau=\beta$. Let $I=\{1,...,|\BB_\gerp-T_\gerp|\}-\{j, j+1\}$. Then 
 $$ p^{n_1(\tau)}L_{T,f_\beta,\tau}(\kappa)=\sum_{i\in I} p^{n_i(\tau)}\ell_{\bar{\sigma}^i\tau} + p^{n_j(\tau)}(p^{n_\beta} \ell_{\sigma^{n_\beta}\beta} - \ell_\beta) \geq 0.$$
 
 \end{proof}
 
 \begin{remark} For a general $T$, set $\calC_T^{{\rm min},0}=\iota_T(\{\kappa \in \calC_T: L_{T,\beta,f_\beta}(\kappa) \geq 0, \forall \beta \in \BB(T)\})$. It is clear that $\calC_T^{{\rm min}} \subseteq \calC_T^{{\rm min},0}$. Proposition \ref{prop: diagonal} shows that $\calC_T^{{\rm min}} =\calC_T^{{\rm min},0}$ if $\tilde{T}=T$.  We expect this equality to be true for all $T\subseteq \BB$.
 \end{remark}

\section{Minimal weights}

 Let $T \subseteq \BB$ and $0 \neq {\bf f} \in H^0(X_T,\omega^\kappa)$. For each $\beta \not\in T$, we let $a_\beta$ be the maximal power of ${\bf f}_\beta(T)$ that divides $\bf f$ (which exists by Lemma 7.2 of \cite{DK2}).   
It follows from part (2) of Proposition \ref{prop:strata} that for any $\beta,\beta' \not\in T$, the strata $X'_T \cap Z({\bf f}_\beta(T))$ and $X'_T \cap Z({\bf f}_{\beta'}(T))$ have no common irreducible components.  This implies that ${\bf f}$ is divisible by $\prod_{\beta \in \BB(T)} {\bf f}_\beta(T)^{a_\beta}$. Hence, we can write 
$${\bf f} ={\bf f}_0\prod_{\beta \not\in T} {\bf f}_\beta(T)^{a_\beta}$$ 
for some ${\bf f}_0 \in H^0(X_T,\omega^{\kappa_0})$,
where $\kappa_0 = \kappa - \sum_{\beta\not\in T} a_\beta f_\beta(T)$.
It follows from the definition of the $a_\beta$ that $\bf f_0$ is not divisible by ${\bf f}_\beta(T)$ for any $\beta \not\in T$, and that $i_T(\kappa_0) \in \calC_T^{T-\rm red}$, where $i_T: \QQ^{\BB} \to \QQ^{\BB - T}$ is defined in (\ref{equation: L-beta}).

\begin{dfn}  For $0 \neq {\bf f} \in H^0(X_T,\omega^\kappa)$, we define the {\em minimal weight} of ${\bf f}$ to be
\[
\Phi({\bf f})=i_T(\kappa_0) \in \ZZ^{\BB- T}.
\]

\end{dfn}

\begin{prop} \label{prop: min weight} Let $T \subseteq \BB$, and $0 \neq {\bf f} \in H^0(X_T,\omega^\kappa)$. Then we have $\Phi({\bf f}) \in \calC^{\rm min}_T$.
\end{prop}
 \begin{proof} Write ${\bf f}={\bf f}_0\prod_{\beta \not\in T} {\bf f}_\beta(T)^{a_\beta}$, where  $a_\beta$ is as above, ${\bf f}_0 \in H^0(X'_T,\delta^{\xi_0}\omega^{\kappa_0})$, and $i_T(\kappa_0)=\Phi({\bf f}) \in \calC_T^{\rm red}$.  If $\Phi({\bf f})\not \in \calC^{\rm min}_T$, then $\kappa_0 \not\in  i_T^{-1}(\calC^{\rm min}_T)=\calC_T \cap (\bigcap_{\beta \in \BB(T)}  \calD_{T,f_\beta})$. Since $\kappa_0 \in \calC_T$, it follows that $\kappa_0 \not\in \calD_{T,f_\beta}$, for some $\beta \in \BB(T)$. Hence,  there exist $\beta \in \BB(T)$, and $\tau \in \BB-(T \cup\{\sigma^{n_\beta}\beta\})$, such that $L_{T,f_\beta,\tau}(\kappa_0)<0$. By Corollary \ref{cor: divisibility}, it follows that ${\bf f}_0$ is divisible by ${\bf f}_\beta(T)$. This contradicts the definition of ${\bf f}_0$.
\end{proof}

 \section{The $\Res_{F/\QQ}\GL_2$ setting} \label{section: GL2}
In this section, we explain how to deduce results analogous to those of preceding sections, but in the context of Shimura varieties for the group $\Res_{F/\QQ}\GL_2$.  In so doing, we also account for the (prime-to-$p$) Hecke action.

\subsection{Hilbert modular varieties and automorphic line bundles} First consider the scheme
$$ \widetilde{X}_K := \coprod_\epsilon X_K^\epsilon, $$
where the disjoint union runs over all $\epsilon \in (\AA_F^{\infty,p})^\times/\det(K^p)(\widehat{\ZZ}^{(p)})^\times$(and $K$, as always, is a sufficiently small open compact subgroup of $\GL_2(\AA_F^\infty)$ with $K_p = \GL_2(\CO_{F,p})$).  Let $\CO_{F,(p)}$ denote the localization of $\CO_F$ at the prime $p$ of $\ZZ$, and $\CO_{F,(p),+}^\times$ the subgroup of $\CO_{F,(p)}^\times$ consisting of totally positive elements.  We have a natural action of $\CO_{F,(p),+}^\times$ on $\widetilde{X}_K$ defined by $\nu\cdot(A/S,\iota,\lambda,\alpha) = (A/S,\iota,\nu\lambda,\alpha)$ for $\nu \in \CO_{F,(p),+}^\times$, sending $X_K^\epsilon(S)$ to $X_K^{\nu\epsilon}(S)$. Note that the action factors through $\CO_{F,(p),+}^\times/(\CO_F^\times \cap K)^2$.  Moreover the resulting action is free for $K$ sufficiently small (see~\cite[Lemma~2.4.1]{DS}), and the quotient $\widetilde{X}_K/\CO_{F,(p),+}^\times$ is smooth of relative dimension $d$ and finite type over $\ZZ_{(p)}$.  Indeed it admits a finite \'etale cover by the disjoint union of the $X_K^\epsilon$, where $\epsilon$ runs over representatives of the finite quotient
$$(\AA_F^{\infty,p})^\times/\CO_{F,(p),+}^\times\det(K^p)(\widehat{\ZZ}^{(p)})^\times = \AA_F^\times/F^\times F_{\infty,+}^\times\det(K)\AA_\QQ^\times.$$
  We let 
$$ X_K' := (\widetilde{X}_K/\CO_{F,(p),+}^\times) \otimes\FF.$$

We also have a natural action of $\CO_{F,(p),+}^\times$ on the line bundle $\omega_\beta$ over $\widetilde{X}_K \otimes \FF$ (induced by the identity on $A$).  If $K$ is sufficiently small (more precisely $p$-neat in the sense of \cite[Def.~3.2.3]{DS}), then this action defines descent data to a line bundle on $X_K'$, which we also denote $\omega_\beta$.  Similarly we write 
$\omega^\kappa$ for the line bundle $\otimes_{\beta\in\BB}\omega_\beta^{k_\beta}$ for $\kappa = \sum k_\beta e_\beta \in \ZZ^\BB$. It is immediate from their definition that the partial Hasse invariants ${\bf h}_\beta$ (on the union of the $X_K^\epsilon \otimes \FF$) descend to sections of $\omega^{h_\beta}$ over $X'_K$, and so also do the stratifications they define.  For $T \subset \BB$, we denote the resulting closed subscheme of $X'_K$ by $X'_{K,T}$, or simply $X'_T$ when $K$ is clear from the context.

Similarly, the line bundles $\wedge^2{\mathcal H}_\beta$ on $\widetilde{X}_K \otimes \FF$ descend to ones on $X_K'$, which we denote by $\delta_\beta$.  Note though that the canonical trivialization defined by $\epsilon_\beta$ does {\em not} descend to $X_K'$. However, we have canonical isomorphisms $\wedge^2{\mathcal H}_{\sigma^{-1}\beta} \stackrel{\sim}{\to} \wedge^2{\mathcal H}_\beta$, descending to isomorphisms $\delta_{\sigma^{-1}\beta}^p \stackrel{\sim}{\lra} \delta_\beta$ for each $\beta \in \BB$ (see for example~\cite[\S3.2.2]{DKS}).  Writing $\delta^\kappa$ for the line bundle $\otimes_{\beta\in\BB}\delta_\beta^{k_\beta}$, where $\kappa = \sum k_\beta e_\beta \in \ZZ^\BB$, we thus obtain a canonical trivialization of $\delta^{h_\beta}$ for each $\beta$.
It follows that $\delta^\kappa$ depends only on the character $(\CO_F/p)^\times \to \FF^\times$ induced by $\kappa$.  More precisely, let $\Delta$ denote the group of characters $(\CO_F/p)^\times \to \FF^\times$, and consider the surjective homomorphism $\ZZ^{\BB} \to \Delta$ defined by $\kappa \to \overline{\kappa}$ (induced by $e_\beta \mapsto \overline{\beta}$).  One sees easily that the kernel is the lattice in $\ZZ^{\BB}$ freely generated by the $h_\beta$, so that $\delta^\kappa$ depends (up to canonical isomorphism) only on $\overline{\kappa} \in \Delta$, and so we denote it $\delta^{\overline{\kappa}}$.  Examining the definition of the section ${\bf b}_\beta(T)$ on (the union over $\epsilon$ of the) $X_T$, we see that it descends to a trivialization of $\delta^{-\overline{\beta}}\omega^{b_\beta}$ on $X_T'$, which we again denote ${\bf b}_\beta(T)$.

\subsection{Hecke actions and equivariance}
Recall from \cite[\S\S2.1.3, 3.1.1]{DKS} that if $K_1$ and $K_2$ are sufficiently small open subgroups of $\GL_2(\AA_{F}^\infty)$ as above, and $g \in \GL_2(\AA_{F}^{\infty,p})$ is such that $g^{-1}K^p_2g \subset K^p_1$, then we have a morphism $\rho_g:X_{K_2}' \to X_{K_1}'$ of schemes over $\FF$, and isomorphisms $\rho_g^*\omega_\beta \stackrel{\sim}{\to} \omega_\beta$ and $\rho_g^*\delta_\beta \stackrel{\sim}{\to} \delta_\beta$ of line bundles over $X_{K_2}$, giving rise to an action of $\GL_2(\AA_{F}^{\infty,p})$ on
$$\varinjlim_{K^p} H^0(X'_K, \delta^\xi\omega^\kappa)$$
for any $\xi \in \Delta$, $\kappa \in \ZZ^{\BB}$.  Identifying $H^0(X'_K,\delta^\xi\omega^\kappa)$ with the $K^p$ invariants under this action, we thus also obtain an action of commuting Hecke operators $T_v$ and $S_v$ on $H^0(X'_K, \delta^\xi\omega^\kappa)$ for all primes $v$ of $F$ such that $\GL_2(\CO_{F,v})\subset K^p$.  Furthermore, the partial Hasse invariants ${\bf h}_\beta$ are invariant under the action of $\GL_2(\AA_F^{\infty,p})$, so the morphisms $\rho_g$ restrict to morphisms $\rho_g:X_{K_2,T}' \to X_{K_1,T}'$ and give rise to an action of $\GL_2(\AA_{F}^{\infty,p})$ on
$$\varinjlim_{K^p} H^0(X'_{K,T}, \delta^\xi\omega^\kappa).$$
We thus also obtain Hecke operators $T_v$ and $S_v$ on $H^0(X'_{T}, \delta^\xi\omega^\kappa)$ for all $K$ and $v$ as above, commuting with each other and with the injective maps
$$H^0(X'_T, \delta^\xi\omega^\kappa) \longrightarrow
 H^0(X'_T, \delta^\xi\omega^{\kappa+h_\beta})$$
defined by multiplication by ${\bf h}_\beta$ for $\beta \not\in T$.  Similarly, the isomorphisms
$$H^0(X'_T, \delta^\xi\omega^\kappa) \stackrel{\sim}{\longrightarrow}
 H^0(X'_T, \delta^{\xi-\overline{\beta}}\omega^{\kappa+b_\beta})$$
defined by multiplication by ${\bf b}_\beta(T)$ for $\beta \in T$ are Hecke-equivariant.

We claim also that the sections ${\bf h}_{\beta}^{\beta'}(T)$ of Proposition~\ref{prop: sections} descend to Hecke-equivariant sections over $X'_T$.  Indeed it follows exactly as in \cite[\S5.3]{DKS} or \cite[\S6.5.2]{M} that the isomorphisms constructed in the proof of \cite[Thm.~5.2]{TX} are Hecke-equivariant in the following sense.  Firstly, implicit in the construction of the isomorphism $\iota_{T'}$ is a choice of isomorphism $\alpha_v:\GL_2(F_v) \stackrel{\sim}{\to} B^\times_{S(T),v}$ for all finite $v \not\in S(T)$ inducing an isomorphism $\GL_2(\AA_F^{\infty,p}) \cong (B_{S(T)}\otimes_F\AA_F^{\infty,p})^\times$. Furthermore, the isomorphisms $\alpha_v$ can be chosen (dependently on $\epsilon$) so that the level $K_T$ in the notation of \S\ref{section: morphisms} is $K'_T \cap G(\AA^{\infty,p})$, where
$$K'_{T,p} = \prod_{\gerp \in S(T)} \CO_{B,\gerp}^\times
\prod_{\gerp \in {\rm{Iw}}(T)} \alpha_{\gerp}(K_0(\gerp))
\prod_{\gerp \not\in S(T)\cup \rm{Iw}(T)} \alpha_{\gerp}(\GL_2(\CO_{F,\gp})),$$
and the union over $\epsilon$ of the isomorphisms $\iota_{T'}$ descends to an isomorphism  $$\xymatrix{X'_{T'}  \ar[r]^{\iota'_{T'}\ \ \ \ \ \ \ \  }_{\cong\ \ \ \ \ \ \ \ } & \Pi_{_{\beta \in \tilde{T}-T'}} \PP(\calV_{\beta}) \ar[r]^{\ \ \pi'_{T'}} &X_{G',K'_{T}}}  $$
for a collection of rank two vector bundles $\calV_\beta$ on the Shimura variety $X_{G',K'_T}$ of level $K'_T$ for the group $G' = B_{S(T)}^\times$.  We thus obtain morphisms of strata $\pi'_{T,T'}:X'_T \to X'_{T'}$ exactly as in \S\ref{section: morphisms}. Note also that for $K_1$, $K_2$ and $g$ as above, we have $g_T^{-1}K'^p_{2,T}g_T \subset K'^p_{1,T}$ (writing $g_T$ for $\alpha^p(g)$), and hence a morphism $\rho_{g_T}:X_{G',K'_{2,T}} \to X_{G',K'_{1,T}}$ of schemes over $\FF$.  Furthermore, the construction of the vector bundles $\calV_\beta$ yields isomorphisms $\rho_{g_T}^*\calV_\beta \stackrel{\sim}{\to} \calV_\beta$, and hence a morphism
$$ \prod_{\beta\in \widetilde{T} - T'} \PP_{X_{G',K'_{2,T}}}(\calV_\beta) \lra \prod_{\beta\in \widetilde{T} - T'} \PP_{X_{G',K'_{1,T}}} (\calV_\beta)$$
over $\rho_{g_T}$.  The argument of \cite[\S5.3]{DKS} shows that the isomorphisms $\iota_{T'}$ can be chosen (for all sufficiently small $K$) so that the resulting isomorphism $\iota'_{T'}$ is compatible with the above morphism of projectivizations for all $g \in \GL_2(\AA_F^{\infty,p})$ (such that $g^{-1}K_2 g \subset K_1$).  In particular, it follows that the morphisms $\pi'_{T,T'}$ between strata are compatible with the morphisms $\rho_g$.
Furthermore, for $\beta \in \BB - \widetilde{T}$, the line bundles $\eta_\beta$ on (the union over $\epsilon$ of) the $X_{G,K_T}$ descend to ones on $X_{G',K'_T}$ such that $\iota'^*_{T'}\pi'^*_{T'}\eta_\beta \cong \omega_\beta$, so we obtain isomorphisms $\pi'^*_{T,T'}\omega^\kappa \cong \omega^\kappa$ exactly as in Lemma~\ref{lemma: pullback} and sections
$${\bf h}_{\beta}^{\beta'}(T) \in H^0(X_T',\omega^{h_\beta^{\beta'}})$$
exactly as in Proposition~\ref{prop: sections}.  Finally, for $K_1$, $K_2$ and $g$ as above, we have canonical isomorphisms $\rho_{g_T}^*\eta_\beta \cong \eta_\beta$, which we see (as in \cite[\S5.4]{DKS} or \cite[\S6.6]{M}) are compatible under $\iota'^*_{T'}\pi'^*_{T'}\eta_\beta \cong \omega_\beta$ with the isomorphims $\rho_g^*\omega_\beta \cong \omega_\beta$.  It follows that the sections ${\bf h}_{\beta}^{\beta'}(T)$ are invariant under the action of $\GL_2(\AA_F^{\infty,p})$.

Recall that for $\beta\not\in T$, we let 
$$f_\beta(T) = \left\{\begin{array}{rl}
h^{\beta}_{\sigma^{n_\beta}\beta},& \mbox{if $\beta\not\in \widetilde{T}$};\\
-b^{\beta}_{\sigma^{n_\beta}\beta},& \mbox{if $\beta\in \widetilde{T}-T$}\end{array}\right.$$
(where slightly extend the definition of $n_\beta$ to be $|\BB_\gp|$ in the case $\BB_\gp = \widetilde{T}_\gp$).
We now define ${\bf f}_\beta(T)$ exactly as before, but as sections on $X'_T$, so that
\begin{itemize}
\item if $\beta \not\in \widetilde{T}$, then ${\bf f}_\beta(T) ={\bf h}^\beta_{\sigma^{n_\beta}\beta} \in H^0(X'_T, \omega^{f_\beta(T)})$ if $\beta \not\in \widetilde{T}$, 
\item and if $\beta \in \widetilde{T}-T$, then ${\bf f}_\beta(T) ={\bf h}^{\sigma\beta}_{\sigma^{n_\beta}\beta}{\bf b}_{\sigma\beta}(T)^{-p^{n_\beta - 1}} \in H^0(X'_T, \delta^{\sigma^{n_\beta}\overline{\beta}}\omega^{f_\beta(T)})$ if $\beta \in \widetilde{T}-T$.
\end{itemize}
Note that by the discussion above, the ${\bf f}_\beta(T)$ are invariant under the action of $\GL_2(\AA_F^{\infty,p})$.  In particular, we have the following:
\begin{proposition} \label{prop: Hasse-like} Suppose that $K = K_pK^p$ is a sufficiently small open subgroup of $\GL_2(\AA_F^{\infty})$ such that $K_p = \GL_2(\CO_{F,p})$, and $v$  is a prime of $F$ such that $\GL_2(\CO_{F,v}) \subset K^p \subset \GL_2(\AA_F^{\infty,p})$.  If $\beta \not\in T \subset \BB$, then
\begin{enumerate}
\item $T_v{\bf f}_\beta(T) = ({\rm Nm}_{F/\QQ}v + 1) {\bf f}_\beta(T)$ and $S_v {\bf f}_\beta(T) = {\bf f}_\beta(T)$,
\item the map 
$$H^0(X'_T,\delta^\xi\omega^\kappa) \stackrel{f_\beta(T)\cdot}{\lra} H^0(X'_T,\delta^{\xi'}\omega^{\kappa+f_\beta(T)})$$
commutes with the operators $T_v$ and $S_v$,
where $\xi' = \xi$ or $\xi + \sigma^{n_\beta}\overline{\beta}$ according to whether or not $\beta\not\in\widetilde{T}$. 
\end{enumerate}
\end{proposition}

\subsection{Cones of weights}
Fix $T \subset \BB$ and a sufficiently small open compact subgroup $K$ of $\GL_2(\AA_F^\infty)$ containing $\GL_2(\CO_{F,p})$.  Consider the cone of weights 
$$C(X'_T,\delta^\bullet\omega^\bullet)= \{\,(\lambda,\kappa) \in \ZZ^\BB \times \ZZ^\BB\,:\, H^0(X_T',\delta^\lambda\omega^\kappa)\neq 0 \,\}$$
and let $\mathcal{C}(X'_T,\delta^\bullet\omega^\bullet)$ be its $\QQ^{\ge 0}$-saturation.  We have the following immediate consequence of Theorem~\ref{thm:main} and the discussion above.
\begin{cor} \label{cor:main} We have $\mathcal{C}(X'_T,\delta^\bullet\omega^\bullet)= \QQ^{\BB} \times \mathcal{D}_T$.
\end{cor}
\begin{proof} Suppose that $0 \neq {\bf f} \in H^0(X_T',\delta^\lambda\omega^\kappa)\neq 0 \,\}$.
For each $\epsilon \in (\AA_F^{\infty,p})^\times$, let 
$${\bf f}^\epsilon \in H^0(X^\epsilon_T, \delta^{\lambda}\omega^{\kappa}) \cong H^0(X^\epsilon_T, \omega^{\kappa})$$
 denote the pull-back of ${\bf f}$ via the finite \'etale morphism $X_T^\epsilon \to X_T'$ (where $X_T^\epsilon$ is the closed subscheme of $X_K^\epsilon\otimes\FF$ defined by the vanishing of the ${\bf h}_\beta$ for $\beta\in T$). 
Since ${\bf f} \neq 0$, we have ${\bf f}^\epsilon \neq 0$ for some $\epsilon$, and hence $\kappa \in \mathcal{C}_T = \mathcal{D}_T$ by Theorem~\ref{thm:main}.  It follows that $\mathcal{C}(X'_T,\delta^\bullet\omega^\bullet)\subset \QQ^{\BB} \times \mathcal{D}_T$.

For the opposite inclusion, note that $C(X'_T,\delta^\bullet\omega^\bullet)$ contains each of the following:
\begin{itemize}
\item $\pm(h_\beta,0)$ for all $\beta\in \BB$ (from the canonical trivializations of $\delta^{h_\beta}$);
\item $\pm(-e_\beta,b_\beta)$ for all $\beta \in T$ (from the trivializations ${\bf b}_\beta(T)$);
\item $(0,f_\beta(T))$ for all $\beta\notin \widetilde{T}$ (from the sections ${\bf f}_\beta(T)$);
\item $(e_{\sigma^{n_\beta}\beta},f_\beta(T))$ for all $\beta\in \widetilde{T} - T$ (from the sections ${\bf f}_\beta(T)$).
\end{itemize}
Since the $\QQ^\BB = \langle h_\beta : \beta \in \BB \rangle$ and
$\mathcal{D}_T = \langle f_\beta(T) : \beta \not\in T \rangle
 + \langle \pm b_\beta : \beta \in T \rangle$, it follows that 
$ \QQ^{\BB} \times \mathcal{D}_T\subset \mathcal{C}(X'_T,\delta^\bullet\omega^\bullet)$.
\end{proof}

\begin{cor} \label{cor:vanishing} Suppose that $\xi \in \Delta$ and $\kappa \in \ZZ^{\BB}$.  If $\kappa\not\in \mathcal{D}_T$, then $H^0(X'_T,\delta^\xi\omega^\kappa) = 0$.
\end{cor}
\begin{proof} This is immediate from Corollary~\ref{cor:main}.
\end{proof}

We may also define a stratification on $X'_T$ as in \S\ref{subsection: stratification}: For $J \subset \BB - T$, we let
$$X'_{T,J}=X'_T \cap \big(\bigcap_{\beta \in J} Z({\bf f}_\beta(T))\big),\quad\mbox{and}\quad
W'_{T,J}=X'_{T,J}-{\bigcup}_{\beta \not \in J} X'_{T,\{\beta\}}.$$
Since the line bundles $\delta_\beta$ are torsion, the same proofs as before carry over to give the following: 
\begin{prop} \label{prop:strata'}
For each $J \subset \BB - T$, we have
\begin{enumerate}
\item $W'_{T,J}$ is non-empty, quasi-affine and equidimensional of dimension $\BB - (T \cup J)$;
\item $X'_{T,J} = \overline{W}'_{T,J} = \bigcup_{J' \supset J} W'_{T,J'}$.
\end{enumerate}
\end{prop}

 Finally, we translate our results on minimal weights to the setting of $\Res_{F/\QQ} \GL_2$.  Suppose that $\kappa \in \ZZ^{\BB}$, $\xi \in \Delta$, $0 \neq {\bf f} \in H^0(X'_T,\delta^\xi\omega^\kappa)$, and for each $\beta \not\in T$, we let $a_\beta$ be the maximal power of ${\bf f}_\beta(T)$ that divides $\bf f$ (which again exists by \cite[Lemma~7.2]{DK2}).   
By part (1) of Proposition \ref{prop:strata'}, the strata $X'_T \cap Z({\bf f}_\beta(T))$ and $X'_T \cap Z({\bf f}_{\beta'}(T))$ have no common irreducible components for distinct $\beta,\beta'\not\in T$, so ${\bf f}$ is divisible by $\prod_{\beta \in \BB(T)} {\bf f}_\beta(T)^{a_\beta}$. We can therefore write 
$${\bf f} ={\bf f}_0\prod_{\beta \not\in T} {\bf f}_\beta(T)^{a_\beta}$$ 
for some ${\bf f}_0 \in H^0(X'_T,\delta^{\xi_0}\omega^{\kappa_0})$,
where $\xi_0 = \xi - \sum_{\beta\in \widetilde{T}-T} a_\beta\sigma^{n_\beta}\overline{\beta}$, $\kappa_0 = \kappa - \sum_{\beta\not\in T} a_\beta f_\beta(T)$, and ${\bf f}_0$ is not divisible by ${\bf f}_\beta(T)$ for any $\beta \not\in T$.

\begin{dfn}  For $0 \neq {\bf f} \in H^0(X'_T,\delta^\xi\omega^\kappa)$, we define the {\em minimal weight} of ${\bf f}$ to be
\[
\Phi({\bf f})=i_T(\kappa_0) \in \ZZ^{\BB - T},
\]
where $i_T: \QQ^{\BB} \to \QQ^{\BB - T}$ is defined in (\ref{equation: L-beta}), so in particular, $\Phi({\bf f}) \in i_T(\calC_T) = \calC_T^{T-\rm red}$.
\end{dfn}

\begin{prop} \label{prop: min weight'} If $0 \neq {\bf f} \in H^0(X'_T,\delta^\xi\omega^\kappa)$, then $\Phi({\bf f}) \in \calC^{\rm min}_T$.
\end{prop}
 \begin{proof} Write ${\bf f}={\bf f}_0\prod_{\beta \not\in T} {\bf f}_\beta(T)^{a_\beta}$ as above, so ${\bf f}_0 \in H^0(X'_T,\delta^{\xi_0}\omega^{\kappa_0})$ and $i_T(\kappa_0)=\Phi({\bf f})$.  As in the proof of Proposition~\ref{prop: min weight}, if $\Phi({\bf f})\not \in \calC^{\rm min}_T$, then $L_{T,f_\beta,\tau}(\kappa_0)<0$ for some $\beta \in \BB(T)$ and  $\tau \in \BB-(T \cup\{\sigma^{n_\beta}\beta\})$.

For each $\epsilon \in (\AA_F^{\infty,p})^\times$, let ${\bf f}_0^\epsilon \in H^0(X^\epsilon_T, \delta^{\xi_0}\omega^{\kappa_0})$ denote the pull-back of ${\bf f}_0$ to $X_T^\epsilon$ (as in the proof of Corollary~\ref{cor:main}).  Since $L_{T,f_\beta,\tau}(\kappa_0)<0$, it follows from Corollary \ref{cor: divisibility} that ${\bf f}_0^\epsilon$ is divisible by (the pull-back of) ${\bf f}_\beta(T)$. Since this holds for all $\epsilon$, it follows that the pull-back of ${\bf f}_0$ to $\coprod_\epsilon X_T^\epsilon$ has the form ${\bf f}_1{\bf f}_\beta(T)$ for some 
$${\bf f}_1 \in H^0(\coprod_\epsilon X^\epsilon_T, \delta^{\xi_1}\omega^{\kappa_1}),$$
where $\xi_1 = \xi$ or $\xi_0 - \sigma^{n_\beta}\overline{\beta}$ and $\kappa_1 = \kappa_0 - f_\beta(T)$.  Since ${\bf f}_0$ and ${\bf f}_\beta(T)$ descend to sections over $X'_T$, it follows that so does ${\bf f}_1$, contradicting the definition of ${\bf f}_0$.
 \end{proof}

Combining Propositions~\ref{prop: Hasse-like} and~\ref{prop: min weight} immediately gives the following:
\begin{cor} Let $S$ be a set of primes containing all $v|p$ and all $v$ such that $\GL_2(\CO_{F,v}) \not\subset K$.  Suppose that ${\bf f} \in H^0(X'_T,\delta^\xi\omega^\kappa)$ is an eigenform for the Hecke operators $T_v$ and $S_v$ for all $v \not\in S$. Then there is an eigenform ${\bf f}' \in H^0(X'_T,\delta^{\xi'}\omega^{\kappa'})$ with the same eigenvalues as ${\bf f}$ for $T_v$ and $S_v$ for all $v\not\in S$, and such that
$i_T(\kappa') \in \calC^{\rm min}_T$.
\end{cor}

\end{document}